\documentclass{siamart1116}
\usepackage{amsmath, amssymb, bm, mathrsfs}
\usepackage{cite, color, graphicx}

\usepackage{lipsum}
\usepackage{amsfonts}
\usepackage{graphicx}
\usepackage{epstopdf}
\usepackage{algorithmic}
\ifpdf
  \DeclareGraphicsExtensions{.eps,.pdf,.png,.jpg}
\else
  \DeclareGraphicsExtensions{.eps}
\fi

\numberwithin{theorem}{section}

\newcommand{\TheTitle}{LQ-optimal Sample-data Control under Stochastic Delays:
Gridding Approach for Stabilizability and Detectability} 
\newcommand{\TheAuthors}{Masashi Wakaiki, Masaki Ogura,
	and Jo\~ao P. Hespanha}

\headers{LQ-optimal Control under Stochastic Delays}{\TheAuthors}

\title{{\TheTitle}\thanks{Submitted to the editors DATE.
This paper was partially presented at the American Control Conference 2017 \cite{WakaikiACC2017}.
\funding{
	This material is based upon work 
	supported by JSPS KAKENHI Grant Number 17K14699 and 
	the National Science Foundation 
	under Grants No. CNS-1329650 and EPCN-1608880.}}}

\author{
  Masashi Wakaiki\thanks{Graduate School of System Informatics, Kobe University,  Hyogo 657-8501, Japan
    (\email{wakaiki@ruby.kobe.u.ac.jp}).}
  \and
  Masaki Ogura\thanks{Division of Information Science, 
  	Graduate School of Science and Technology, Nara Institute of Science and Technology, Nara 630-0192, Japan
    (\email{oguram@is.naist.jp}).}
  \and
 Jo\~ao P. Hespanha\thanks{Center for Control,
 	Dynamical-systems and Computation (CCDC), University of California,
 	Santa Barbara, CA 93106, USA
 	(\email{hespanha@ucsb.edu}).}
}

\usepackage{amsopn}
\DeclareMathOperator{\diag}{diag}

\ifpdf
\hypersetup{
  pdftitle={\TheTitle},
  pdfauthor={\TheAuthors}
}
\fi

\newtheorem{problem}[theorem]{Problem}
\newtheorem{assumption}[theorem]{Assumption}

\newtheorem{remark}[theorem]{Remark}

\begin{document}

\maketitle

\begin{abstract}
We solve a linear quadratic optimal control problem 
for sampled-data systems with stochastic delays.
The delays are stochastically determined by the last few delays.
The proposed optimal controller can be efficiently computed 
by iteratively solving a Riccati difference equation,
provided that a discrete-time Markov jump system equivalent to
the sampled-data system is stochastic stabilizable and detectable. 
Sufficient conditions for these notions are provided in the form of linear matrix inequalities, 
from which stabilizing controllers and state observers can be constructed.
\end{abstract}



\section{Introduction}
Communication delays occur in networked control systems due to
signal processing and congestion in busy channels.
Such delays are generally time-varying, and 
if their range is large, control methods 
developed for systems with constant delays in \cite{Curtain1995,Foias1996} may not be suitable.
With a rapid development of communication technologies,
control under time-varying delays
has received extensive attention over recent decades,
as surveyed in \cite{Yang2006, Hespanha2007}.

One approach to compensate for time-varying delays is 
the virtual time-delay method developed, e.g.,  for bilateral control \cite{Kosuge1996}.
In this method, the maximum value of delays is assumed to be known, and
control signals are updated when the maximum time of delays has passed. This
method keeps the apparent delays constant but may degrade the performance of networked control systems 
if the maximum delay is quite larger than the average delay.
Another approach is to measure delays
by time-stamped messages
and exploit these measurements in the control algorithms, as in Fig.~\ref{fig:closed_loop}.
An example of this scenario can be found in inter-area power systems \cite{Safaei2014}.
Controllers using time-stamp information in this way are delay-dependent, and
stabilization by such controllers has been studied 
in \cite{Hirche2008,Hetel2011,Kruszewski2012,
	Safaei2014,Zhang2005, Huang2008, Shi2009_stochastic_delay,
	Sun2011, Demirel2013} and references therein.
Time-stamped messages are also used for 
linear quadratic (LQ) control in \cite{Nilsson1997,Nilsson1998, Shousong2003, Kordonis2014}
and for model predictive control in \cite{Srini2004} under stochastically time-varying delays.

\begin{figure}[tb]
	\centering
	\includegraphics[width = 5.5cm]{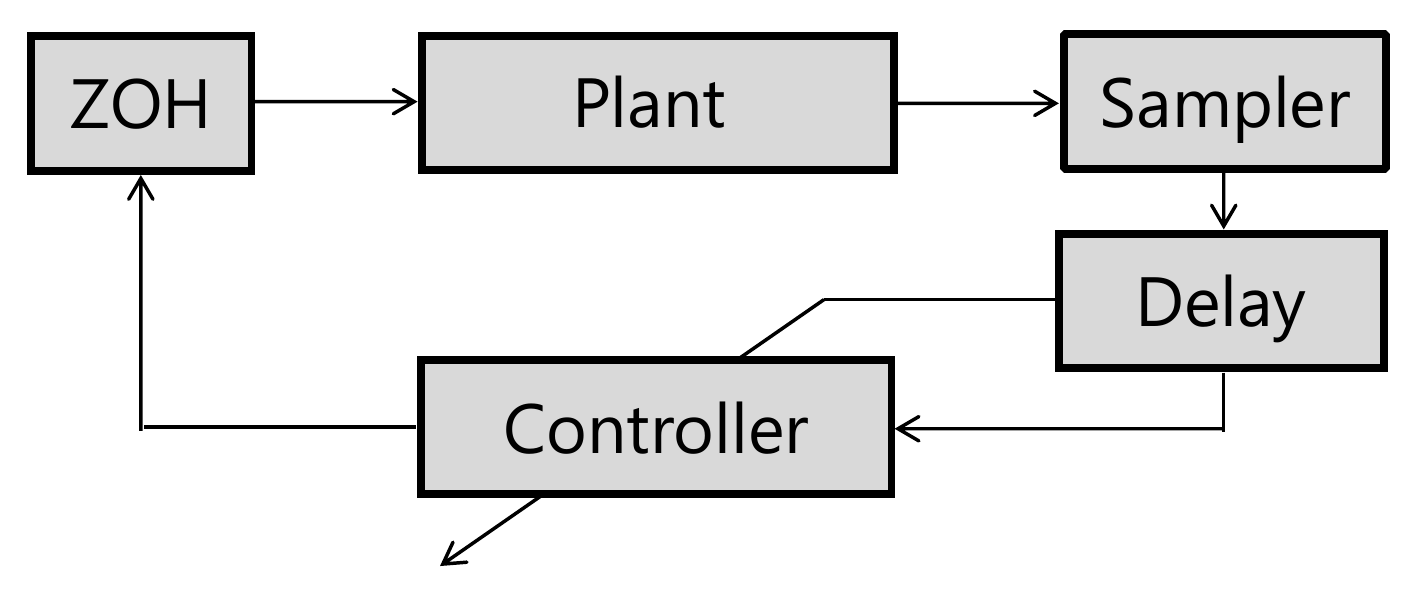}
	\caption{Closed-loop system.}
	\label{fig:closed_loop}
\end{figure}

In addition to the earlier studies mentioned above,
the authors in \cite{Xu2012,Kobayashi2014, Ceja2015,Wang2016,Komanovsky2001}  have developed 
design methods of LQ controllers
for scenarios where the
measurements of random delays are not available.
However, most of those syntheses of optimal controllers 
require online computation and 
are based on assumptions that
the delays can take only finitely many values or 
be independent and identically-distributed random variables. 
In practice, 
communication delays take a continuum of values and
need to be modeled by more general stochastic processes such as Markov processes in \cite{Hermenier2009,Ghanaim2011,Cong2010}.
A notable exception can be found in \cite{Kordonis2014}, in which the authors have
presented an offline computation method of delay-dependent LQ controllers
for systems with
continuous-valued Markovian
delays. The formulation in \cite{Kordonis2014}
requires a solution of a nonlinear vector integral equation called the Riccati integral equation and
ignores the intersample behavior of the closed-loop system.

In this paper, we study delay-dependent LQ
control for sampled-data linear systems.
The advantages of the proposed method are twofold.
First, our delay model is more general than that in the above previous studies.
Indeed, in the model we consider, the present
delay is determined by the last few delays like in an
autoregressive models (see, e.g., Chapter 9 of \cite{Percival1993} for autoregressive models), and hence 
our delay model belongs to a class of higher-order Markov models.
Second, we can efficiently compute an LQ control law
that takes into account the intersample behavior. This controller is obtained 
by iteratively solving a Riccati difference equation.

A key step in the construction 
of the controller consists of 
reducing the original sampled-data problem into
an LQ problem for a  discrete-time Markov jump system
whose jumps are modeled by a Markov chain taking values in a general Borel space.
In \cite{Costa2015}, the reduced LQ problem has been solved under
the assumption that the plant and the LQ criterion satisfy 
appropriately defined
stochastic stabilizability and detectability notions.
However, there has been relatively little work on the test of these properties.
To use the results on stochastic stability in \cite{Costa2014},
we obtain novel  sufficient conditions 
for stochastic stabilizability and detectability 
in terms of linear matrix inequalities (LMIs).
From these results, we can also construct stabilizing controllers and state observers.
The proposed method is inspired by the gridding methods 
for establishing the stability 
of networked control systems with aperiodic sampling and time-varying delays
in \cite{Fujioka2009,Oishi2010, Donkers2011,Donkers2012_stochastic, Hetel2011}.
Moreover, we show that the sufficient conditions can be arbitrarily tight 
under certain assumptions.


The remainder of the paper is organized as follows. 
We provide the problem statement in Section~2.
Section~3 is devoted to reducing our optimal control problem to an LQ problem
for discrete-time Markov jump systems.
In Section~4, we recall the general results in \cite{Costa2015}  on
the equivalent discrete-time LQ problem.
Section~5 addresses the derivation of sufficient conditions for stochastic
stabilizability and detectability.
In Section~6,
we illustrate the proposed method with a numerical simulation of a batch reactor.

\subsubsection*{Notation}
Let $\mathbb{Z}_+$, $\mathbb{R}^{n\times m}$, and $\mathbb{C}^{n \times m}$ denote the set of nonnegative integers and 
the sets of real and complex matrices with size $n \times m$, respectively.
For a real matrix $M$, let us denote its transpose by $M^{\top}$.
The Euclidean norm of $v \in \mathbb{R}^{n}$ is
denoted by $\|v\| := (v^\top v)^{1/2}$ and
the corresponding induced norm of a matrix $M \in \mathbb{R}^{m\times n}$ by
$\|M\|  := \sup \{  \|Mv \|:~
v\in \mathbb{R}^{n},~\|v\| = 1 \}$.
For simplicity, 
we write a partitioned symmetric
matrix 
\[\begin{bmatrix}
Q & W \\ W^\top & R
\end{bmatrix}
\text{~~as~~}
\begin{bmatrix}
Q & W \\ \star & R
\end{bmatrix}.
\]

Let $(\mathcal{M}, \mathcal{B}(\mathcal{M}))$ be a Borel space, that is, 
$\mathcal{M}$ be a Borel subset of a complete and separable metric space and
$ \mathcal{B}(\mathcal{M})$ be its Borel $\sigma$-algebra.
In this article, $\mathcal{M}$ is a compact subset of $\mathbb{R}^{p}$ except in Section 4.
For a $\sigma$-finite measure $\mu$ on $\mathcal{M}$, we denote by
$\mathbb{H}^{n \times m}_{1}$ 
the space of matrix-valued functions 
$P(\bullet):\mathcal{M} \to　\mathbb{R}^{n \times m}$
that are measurable and integrable in $\mathcal{M}$, and
similarly, by
$\mathbb{H}^{n \times m}_{\sup}$ the space of matrix-valued functions
$P(\bullet):\mathcal{M} \to \mathbb{R}^{n \times m}$ 
that are measurable and essentially bounded in $\mathcal{M}$.
For $P \in \mathbb{H}^{n \times m}_{\sup}$, we define a norm
$\|P\|_{\infty}$ by the essential supremum of the function $\|P\|:\mathcal{M} \to [0,\infty)$.
For simplicity, we will write 
$\mathbb{H}^{n}_{1} := \mathbb{H}^{n\times n}_{1}$, and
$\mathbb{H}^{n}_{\sup} := \mathbb{H}^{n \times n}_{\sup}$,
\[\mathbb{H}^{n+}_{\sup} := \{P \in \mathbb{H}_{\sup}^n:P(\phi)\geq 0\quad 
\text{for $\mu$-almost every $\phi \in \mathcal{M}$} \}.\]
Additionally, we denote by $\mathbb{H}^{n}_{1,\mathbb{C}}$
the space of matrix-valued functions
$P(\bullet):\mathcal{M} \to \mathbb{C}^{n \times m}$ 
that are measurable and integrable in $\mathcal{M}$
and by
$\mathbb{H}^{n+}_{\sup,\mathbb{C}}$
the space of matrix-valued functions
$P(\bullet):\mathcal{M} \to \mathbb{C}^{n \times m}$ 
that are measurable and essentially bounded in $\mathcal{M}$ and
satisfy $P(\phi) \geq 0$ for $\mu$-almost every $\phi \in \mathcal{M}$.
For a bounded linear operator $T$ on a Banach space,
let $r_{\sigma}(T)$ denote the spectral radius of $T$.

\section{Problem Statement}
Consider the following linear continuous-time plant:
\begin{equation}
\label{eq:plant_dynamics}
\dot x(t) = A_c x(t) + B_c u(t),\qquad  x(0) = x_0
\end{equation}
where $x(t) \in \mathbb{R}^n$ and
$u(t) \in \mathbb{R}^m$
are the state and the input of the plant.
This plant is connected to a controller through 
a time-driven sampler with period $h >0$ and an event-driven zero-order hold as
shown in Fig.~\ref{fig:closed_loop}.

The state $x$ is measured at each sampling time $t=kh$ ($k\in\mathbb{Z}_+$), 
and the controller receives the sampled state $x(kh)$ 
at time $t = kh + \tau_k$, where $\tau_k > 0$ is a 
sensor-to-controller delay.
We assume that the delay $\tau_k$ 
becomes known to the controller at the time $t=kh+\tau_k$ when the sampled state $x(kh)$ arrives.
One way to measure the delays
is to mark every output of the sampler with a time-stamp
and then to compute the difference between the value of the time-stamp and the present time
of a clock in the controller.
Through 
the zero-order hold,
the discrete-time signal $u_k$ generated from the controller is
transformed  to 
the continuous-time signal
\begin{equation}
\label{eq:input}
u(t) = 
\begin{cases}
u_{-1} & 0 \leq t < \tau_0 \\
u_k &  kh+\tau_k \leq t < (k+1)h + \tau_{k+1},~k\in \mathbb{Z}_+.
\end{cases}
\end{equation}
where $u_{-1}$ is an initial state of the zero-order hold.


Throughout this paper, we fix the probability space $(\Omega, \mathcal{F}, P)$.
We assume that the delays $\{\tau_k:k\in \mathbb{Z}_+\}$ 
is smaller than one sampling period and that
the latest delay is stochastically determined by the last few delays.
We specifically assume that the delay sequence $\{\tau_k:k\in \mathbb{Z}_+\}$ 
is a higher-order Markov chain. For some known $p \in \mathbb{N}$,
define a delay vector $\phi_k$ by
\begin{equation}
\label{eq:phi_def}
\phi_k :=
\begin{bmatrix}
\tau_{k} \\ \vdots \\ \tau_{k-p+1}
\end{bmatrix}
\qquad \forall k \in \mathbb{Z}_+,
\end{equation}
where $\tau_{-p+1},\dots, \tau_{-1} < h$ are the time delays
associated with the sampling instants $t = (-p+1)h,\dots,-h$.
\begin{assumption}[Higher-order Markovian delays]
	\label{assump:delay}
	The sequence $\{ \phi_k:k \in \mathbb{Z}_+\}$ in \eqref{eq:phi_def} is 
	a time-homogeneous Markov chain 
	taking values in  $\mathcal{M} := [\tau_{\min}, \tau_{\max}]^p
	\subset [0,h)^p$ 
	and
	having transition probability kernel $\mathcal{G}(\bullet, \bullet)$ with
	a density $g(\bullet, \bullet)$ with respect to a $\sigma$-finite measure $\mu$
	on $\mathcal{M}$, so that for every
	$k \in \mathbb{Z}_+$ and every Borel set
	$\mathcal{B}$ of $\mathcal{M}$, 
	\[
	\mathcal{G}(\phi, \mathcal{B} ) :=
	P(\phi_{k+1} \in \mathcal{B} | \phi_k = \phi)
	= \int_{\mathcal{B}} g(\phi, \ell) \mu(d\ell).
	\]
\end{assumption}

The choice of the dimension $p$ depends on  accuracy of delay models and 
computational cost.
As the dimension $p$ increases, we may obtain more accurate models of delays.
However, a large $p$ requires substantial computational resources for optimal controllers. 
Moreover, the gridding method presented in Section 5 suffers from the curse of dimensionality in
the case of large $p$.


Define 
\begin{equation}
\label{eq:discrete_state}
\xi_0 := 
\begin{bmatrix}
x_0 \\ u_{-1}
\end{bmatrix}.
\end{equation}
Let $\hat \mu$ be a probability measure on $\mathbb R^{n+m} \times \mathcal{M}$.
We assume that 
the pair of the initial state and delay
$(\xi_0,\phi_0)$ has a distribution $\widehat \mu$. 
Define $\widehat \mu_{\mathcal{M}}$ by 
$\widehat \mu_{\mathcal{M}}(\mathcal{B}) :=
\widehat \mu(\mathbb{R}^{n+m} \times \mathcal{B}) = P(\theta_0 \in \mathcal{B})$ for all Borel sets 
$\mathcal{B}$ of $\mathcal{M}$.
We place the following mild assumption on the initial distribution $\widehat \mu$:
\begin{assumption}[Initial distribution]
	\label{assump:initial_dist}
	The initial distribution $\widehat \mu$ of $(\xi_0,\phi_0)$ satisfies
	A1) $E(\|\xi_0\|^2) < \infty$ and A2)
	$\widehat \mu_{\mathcal{M}} $ is absolutely continuous with respect to $\mu$.
\end{assumption}
The assumption of absolute continuity
guarantees the existence of the Radon-Nikodym
derivative of $\widehat \mu_{\mathcal{M}}$.

Let $\{\mathcal{F}_k:k \in \mathbb{Z}_+\}$ denote a filtration, where
$\mathcal{F}_k$ represents the $\sigma$-field generated by
\[
\{u_{-1}, x(0),\phi_0,\dots,x(kh),\phi_k \} = \{\tau_{-p+1},\dots, \tau_{-1}, u_{-1}, x(0),\tau_0,\dots,x(kh),\tau_k \}. 
\]
Set $\mathcal{U}_c$ as the class of
control inputs $u = \{u_k: k \in \mathbb{Z}_+\}$ such that $u_k$ is $\mathcal{F}_k$
measurable and the controlled system \eqref{eq:plant_dynamics} and
\eqref{eq:input} satisfies 
$E(\|x(t)\|^2) \to 0$ as $t\to \infty$ 
and $E(\|u_k\|^2) \to 0$ as $k\to \infty$ for
every initial distribution $\widehat \mu$ satisfying Assumption~\ref{assump:initial_dist}.
For all $u \in \mathcal{U}_c$,
we consider the infinite-horizon continuous-time 
quadratic cost functional $\mathcal{J}_c$
defined by
\begin{align}
\label{eq:Jc_u}
\mathcal{J}_c(\widehat \mu ,u) := 
E\left(
\int^{\infty}_0
x(t)^{\top} Q_c x(t) + u(t)^{\top} R_c u (t) 
dt
\right),
\end{align}
where $Q_c \geq 0$ and $R_c > 0$ are weighting matrices with appropriate dimensions.

In this paper, we study the following LQ problem:
\begin{problem}
	\label{prob:LQ}
	Consider a sampled-data system \eqref{eq:plant_dynamics} and \eqref{eq:input}, and let
	Assumptions \ref{assump:delay} and \ref{assump:initial_dist} hold.
	Find an optimal control law $u^{\rm{opt}} \in \mathcal{U}_c$ that achieves
	$\mathcal{J}_c(\widehat \mu,u^{\rm{opt}}) = \inf_{u \in \mathcal{U}_c} 
	\mathcal{J}_c(\widehat \mu,u)$
	for every initial distribution $\widehat \mu$  satisfying Assumption~\ref{assump:initial_dist}.
\end{problem}

\begin{remark}
	\label{rem:large_delay}
	{ \em
		In this paper, we impose the following two  assumptions on delays:
		(i) The communication channel from the controller to the actuator has no delays; 
		(ii) Delays are smaller than one sampling period $h$.
		If the controller-to-actuator delays $\tau_{ca}$ are deterministic (see, e.g.,
		\cite[Section~2.3.2]{Yang2006} for this situation), then
		we can also apply the proposed method in the presence of the
		controller-to-actuator delays by using the total delays $\tau_k + \tau_{ca}$ instead of 
		the sensor-to-controller delays $\tau_k$. 
		To deal with delays larger than one sampling period $h$,
		we can employ
		the technique presented in Section II of \cite{Cloosterman2009},
		but a stochastic model on delays becomes complicated. Therefore, 
		we here assume that delays are smaller than one sampling period $h$.
	}
\end{remark}

\section{Reduction to Discrete-time LQ Problem}
In this section, 
we transform Problem \ref{prob:LQ}
to an LQ problem of discrete-time Markov jump linear systems.

Consider the sampled-data system \eqref{eq:plant_dynamics} and
\eqref{eq:input}.
We define
\begin{equation}
\label{eq:xi_def}
\xi_k := 
\begin{bmatrix}
x(kh) \\ u_{k-1}
\end{bmatrix}.
\end{equation}
Then the dynamics of $\xi$ 
can be described by
the following discrete-time Markov jump linear system
\begin{equation}
\label{eq:discretized_sys}
\xi_{k+1} = 
A(\phi_k)
\xi_k + 
B(\phi_k)
u_k,
\end{equation}
where, for every vector $\phi \in \mathcal{M}$ whose the first element is given by $\tau$, we define
\begin{subequations}
	\label{eq:A_B_def} 
	\begin{gather}
	A(\phi):=
	\begin{bmatrix}
	A_d & B_d - \Gamma(\tau)  \\
	0 & 0
	\end{bmatrix},\quad
	B(\phi):=
	\begin{bmatrix}
	\Gamma(\tau) \\
	I
	\end{bmatrix}  \\
	A_d := e^{A_ch},~
	B_d := \int^h_0e^{A_cs} B_cds,~
	\Gamma(\tau) := \int^{h-\tau}_0 e^{A_cs} B_c ds . 
	\end{gather}
\end{subequations}
By definition, the matrices $A$ and $B$ satisfy $A \in \mathbb{H}^{n}_{\sup} $
and $B \in \mathbb{H}^{n \times m}_{\sup} $.
This delay-dependent discrete-time system is widely used for 
the analysis of time-delay systems, e.g., in
\cite{Nilsson1997,Nilsson1998, Shousong2003,
	Donkers2011, Donkers2012_stochastic, Hetel2011,Kordonis2014}.

Let $\{\mathcal{F}^d_k:k \in \mathbb{Z}_+\}$ denote a filtration, where
$\mathcal{F}^d_k$ represents the $\sigma$-field generated by
\[
\{\xi_0,\phi_0,\dots,\xi_k,\phi_k \}.
\]
We denote by $\mathcal{U}_d$ 
the discrete-time counterpart of $\mathcal{U}_c$, 
defined
as the class of
control inputs $\{u_k: k \in \mathbb{Z}_+\}$ such that $u_k$ is $\mathcal{F}^d_k$
measurable and $E(\|\xi_k\|^2) \to 0$ as $k \to \infty$ for
every initial distribution $\widehat \mu$.
The following result establishes that these classes of control inputs are equal.
\begin{lemma}
	\label{lem:Ud_Uc}
	For
	the sampled-data system \eqref{eq:plant_dynamics},
	\eqref{eq:input} and its discretization \eqref{eq:discretized_sys},
	we obtain $\mathcal{U}_c = \mathcal{U}_d$.
\end{lemma}
\begin{proof}
	Since the filtrations $\{\mathcal{F}_k\}$ and $\{\mathcal{F}_k^d\}$ are equal by definition, 
	it is enough to prove that
	the following two conditions are equivalent:
	\begin{enumerate}
		\item
		$\displaystyle \lim_{k\to\infty}E(\|\xi_k\|^2) = 0 $.
		\item
		$\displaystyle \lim_{t\to\infty}E(\|x(t)\|^2) = 0$ and
		$\displaystyle \lim_{k \to \infty }E(\|u_k\|^2) = 0.$
	\end{enumerate}
	The statement 2 $\Rightarrow$ 1
	follows directly from the definition of $\xi_k$. 
	To prove the converse statement, we note that for
	the system dynamics \eqref{eq:plant_dynamics} and \eqref{eq:input}, 
	there exist constants $M_1,M_2,M_3>0$ such that
	\begin{equation}
	\label{eq:x_trajectory_bound}
	\|x(kh+\theta)\| \leq M_1 \|x(kh)\| + M_2 \|u_{k-1}\| + M_3 \|u_{k}\|
	\qquad \forall \theta \in [0,h).
	\end{equation}
	Therefore
	\begin{align*}
	\|x(kh+\theta)\|^2 &\leq M_1^2 \|x(kh)\|^2 + M_2^2 \|u_{k-1}\|^2 + M_3^2 \|u_{k}\|^2
	+ 2M_1M_2 \|x(kh)\| \cdot \|u_{k-1}\| \\ &\qquad + 2M_2M_3\|u_{k-1}\|  \cdot \|u_{k}\|
	+ 2M_3M_1 \|u_{k}\|\cdot \|x(kh)\| \qquad \forall \theta \in [0,h).
	\end{align*}
	Since 
	\[
	ab \leq \frac{a^2+b^2}{2}\qquad \forall a,b\geq 0,
	\]
	applying this inequality to the terms $\|x(kh)\| \cdot \|u_{k-1}\|$,
	$\|u_{k-1}\|  \cdot \|u_{k}\|$, and $\|u_{k}\|\cdot \|x(kh)\|$, we obtain
	\[\|x(kh+\theta)\|^2 \leq N_1 \|x(kh)\|^2 + N_2 \|u_{k-1}\|^2 + N_3 \|u_{k}\|^2
	\qquad \forall \theta \in [0,h)
	\]
	for appropriate constants $N_1,N_2,N_3 > 0$.
	It then follows that if $E(\|\xi_k\|^2) \to 0$ as $k \to \infty$,
	then $E(\|x(t)\|^2) \to 0$ as $t \to \infty$,
	which completes the proof.
\end{proof}

Since the integrand $x(t)^{\top} Q_c x(t) + u(t)^{\top} R_c u (t)$ of 
the cost functional $\mathcal{J}_c$ in \eqref{eq:Jc_u} is non-negative
for every $t \geq 0$,
the cost functional $\mathcal{J}_c$ 
can be expressed as the following (discrete-time) summation 
\[
\mathcal{J}_c = E \left( \sum_{k=0}^\infty \mathcal{J}_k \right),
\]
where 
\begin{equation}
\label{eq:J_k_def}
{\cal J}_k := 
\int^{(k+1)h}_{kh}
x(t)^{\top} Q_c x(t) + u(t)^{\top} R_c u (t) 
dt \qquad \forall k \in \mathbb{Z}_+.
\end{equation}
For every vector $\phi \in \mathcal{M}$ whose  first element is given by $\tau$, we define
the matrices $Q$, $W$, and $R$ by
\begin{subequations}
	\label{eq:QWR_def}
	\begin{align}
	Q(\phi) &:=
	\begin{bmatrix}
	Q_{11}(\tau) & Q_{12}(\tau)\\
	\star & Q_{22}(\tau)
	\end{bmatrix} \\
	W(\phi) &:= 
	\int^h_{\tau}
	\begin{bmatrix}
	\alpha(\theta)^{\top} Q_c \gamma(\tau, \theta)  \\
	\beta(\theta)^{\top} Q_c \gamma(\tau,\theta) 
	-
	\gamma(\tau,\theta)^{\top} Q_c \gamma(\tau,\theta) 
	\end{bmatrix}
	d\theta \\
	R(\phi) &:=
	(h-\tau) R_c + 
	\int^h_{\tau}\gamma(\tau,\theta)^{\top} Q_c \gamma(\tau,\theta) d\theta  \\
	Q_{11}(\tau) &:= 
	\int^h_0 \alpha(\theta)^{\top} Q_c \alpha(\theta) d\theta \\
	Q_{12}(\tau) &:=
	\int^h_0 \alpha(\theta)^{\top} Q_c \beta(\theta) d\theta
	-
	\int^h_\tau \alpha(\theta)^{\top} Q_c \gamma(\tau, \theta) d\theta  \\
	Q_{22}(\tau) &:=
	\tau R_c + \int^h_0 \beta(\theta)^{\top} Q_c \beta(\theta) d\theta + 
	\int^h_{\tau} \gamma(\tau,\theta)^{\top} Q_c \gamma(\tau,\theta) d\theta  \\
	&\qquad  - 
	\int^h_{\tau} 
	\left( \beta(\theta)^{\top} Q_c \gamma(\tau,\theta)  
	+
	\gamma(\tau,\theta)^{\top} Q_c \beta(\theta)
	\right) d\theta, \notag 
	\end{align}
\end{subequations}
where the functions $\alpha$, $\beta$, and $\gamma$ are given by
\begin{gather*}
\alpha(\theta) := e^{A_c\theta},\quad 
\beta(\theta) := \int^{\theta}_0 e^{A_cs} B_cds\qquad 
\forall \theta \in [0,h) \\
\gamma(\tau,\theta) := \int^{\theta-\tau}_0 e^{A_cs} B_c ds\qquad
\forall \tau \in [\tau_{\min}, \tau_{\max}],~\forall \theta \in [0,h).
\end{gather*}

The next lemma shows that each $\mathcal{J}_k$ 
is a quadratic form on the state $\xi_k$ and the input $u_k$ of the discrete-time system $\Sigma_d$ in
\eqref{eq:discretized_sys}. 
\begin{lemma}
	\label{lem:reduction_DT_LQ_cost}
	Let $x$ and $\xi$ be the solutions of the sampled-data system \eqref{eq:plant_dynamics} and
	\eqref{eq:input} and of the discrete-time system \eqref{eq:discretized_sys}
	with the initial state $\xi_0$ defined by \eqref{eq:discrete_state}, receptively.
	Then ${\cal J}_k$ defined by \eqref{eq:J_k_def} satisfies
	\begin{equation}
	\label{eq:Jk_QSR}
	\mathcal{J}_k = 
	\begin{bmatrix}
	\xi_k \\ u_k
	\end{bmatrix}^{\top}
	\begin{bmatrix}
	Q(\phi_k) & W(\phi_k) \\ \star & R(\phi_k)
	\end{bmatrix}
	\begin{bmatrix}
	\xi_k \\ u_k
	\end{bmatrix}
	\qquad \forall k \in \mathbb{Z}_+, 
	\end{equation}
	where the matrices $Q$, $W$, and $R$ are defined as in \eqref{eq:QWR_def}.
\end{lemma}
\begin{proof}
	If $0\leq \theta < \tau_k$, then
	\begin{align*}
	x(kh+\theta) 
	&=
	\alpha(\theta) x_k + \beta(\theta) u_{k-1} \\
	u(kh+\theta)
	&=u_{k-1},
	\end{align*}
	and we have
	\begin{align*}
	&x(kh+\theta)^{\top} Q_c x(kh+\theta) + u(kh+\theta)^{\top} R_c u (kh+\theta)  \\
	&~~=
	x_k^{\top} \alpha(\theta)^{\top} Q_c \alpha(\theta) x_k
	+
	2x_k^{\top} \alpha(\theta)^{\top} Q_c \beta(\theta) u_{k-1} +
	u_{k-1}^{\top} \beta(\theta)^{\top} Q_c \beta(\theta) u_{k-1} \\
	&\qquad + 
	u_{k-1}^{\top} R_c u_{k-1}.
	\end{align*}
	On the other hand, if $\tau_k \leq \theta < h$, then
	\begin{align*}
	x(kh+\theta) 
	&=
	\alpha(\theta) x_k 
	+ (\beta(\theta) - \gamma(\tau_k,\theta)) u_{k-1} 
	+ \gamma(\tau_k,\theta) u_{k} \\
	u(kh+\theta)&= u_k.
	\end{align*}
	Hence 
	\begin{align*}
	&x(kh+\theta)^{\top} Q_c x(kh+\theta) + u(kh+\theta)^{\top} R_c u (kh+\theta) \\
	&\quad =
	x_{k}^{\top} \alpha(\theta)^{\top} Q_c \alpha(\theta)x_k + 
	2 x_{k}^{\top} \alpha(\theta)^{\top} Q_c (\beta(\theta) - \gamma(\tau_k,\theta))u_{k-1} \\
	&\qquad+ 
	2 x_{k}^{\top} \alpha(\theta)^{\top} Q_c \gamma(\tau_k,\theta)u_{k} 
	+ 
	u_{k-1}^{\top} (\beta(\theta) - \gamma(\tau_k,\theta))^{\top} Q_c
	(\beta(\theta) - \gamma(\tau_k,\theta))u_{k-1} \\
	&\qquad +
	2u_{k-1}^{\top} (\beta(\theta) - \gamma(\tau_k,\theta))^{\top} Q_c
	\gamma(\tau_k,\theta)u_{k} +
	u_{k}^{\top}\gamma(\tau_k,\theta)^{\top} Q_c
	\gamma(\tau_k,\theta)u_{k} + u_{k}^{\top} R_c u_{k}.
	\end{align*}
	Substituting these equations into
	\begin{align*}
	\mathcal{J}_k &=
	\int^{kh+\tau_k}_{kh}
	\left(
	x(t)^{\top} Q_c x(t) + u(t)^{\top} R_c u (t) 
	\right)dt \\
	&\qquad 
	+
	\int^{(k+1)h}_{kh+\tau_k}
	\left(
	x(t)^{\top} Q_c x(t) + u(t)^{\top} R_c u (t) 
	\right)dt,
	\end{align*}
	we obtain the desired expression \eqref{eq:Jk_QSR}.
\end{proof}

For the discrete-time Markov jump system \eqref{eq:discretized_sys},
we define the infinite-horizon discrete-time 
quadratic cost functional $\mathcal{J}_{d1}$
by
\begin{align}
\label{eq:Jd1}
\mathcal{J}_{d1}(\widehat \mu ,u) := 
\sum_{k=0}^{\infty}
E\left(
\begin{bmatrix}
\xi_k \\ u_k
\end{bmatrix}^{\top}
\begin{bmatrix}
Q(\phi_k) & W(\phi_k) \\ \star & R(\phi_k)
\end{bmatrix}
\begin{bmatrix}
\xi_k \\ u_k
\end{bmatrix}
\right).
\end{align}
Then a solution to Problem \ref{prob:LQ} for sampled-data systems with stochastic delays
can be obtained as a solution to the following problem for discrete-time Markov jump systems:
\begin{problem}
	\label{prob:LQ_dis}
	Consider a discrete-time Markov jump system \eqref{eq:discretized_sys}, and let
	Assumptions \ref{assump:delay} and \ref{assump:initial_dist} hold.
	Find an optimal control law $u^{\rm{opt}} \in \mathcal{U}_d$ that achieves
	$\mathcal{J}_{d1}(\widehat \mu,u^{\rm{opt}}) = \inf_{u \in \mathcal{U}_d} 
	\mathcal{J}_{d1}(\widehat \mu,u)$
	for every initial distribution $\widehat \mu$ satisfying Assumption~\ref{assump:initial_dist}.
\end{problem}

\begin{lemma}
	\label{lem:reduction1}
	A control input 
	$u^{\rm{opt}} $ is a solution to Problem \ref{prob:LQ} if and only if
	$u^{\rm{opt}} $ is also a solution to
	Problem \ref{prob:LQ_dis} where the system matrices $A$, $B$ and 
	the weighting matrices $Q$, $W$, $R$ are
	defined by \eqref{eq:A_B_def} and \eqref{eq:QWR_def}.
\end{lemma}
\begin{proof}
	Since $\mathcal{U}_c = \mathcal{U}_d$
	from Lemma \ref{lem:Ud_Uc}, it follows that 
	$u^{\rm{opt}}  \in \mathcal{U}_c$ if and only if $u^{\rm{opt}}  \in \mathcal{U}_d$.
	Let $x$ and $\xi$ be the solutions of the sampled-data system \eqref{eq:plant_dynamics} and
	\eqref{eq:input} and of the discrete-time system 
	\eqref{eq:discretized_sys} with the initial state $\xi_0$ defined by \eqref{eq:discrete_state}, receptively.
	Since 
	\[
	x(t)^{\top} Q_c x(t) + u(t)^{\top} R_c u (t) \geq 0\qquad \forall t \geq 0,
	\]
	we obtain
	\begin{align*}
	\mathcal{J}_c(\widehat \mu,u) 
	&=
	E
	\left(
	\int^{\infty}_{0}
	x(t)^{\top} Q_c x(t) + u(t)^{\top} R_c u (t) 
	dt
	\right) \\
	&=
	E	
	\sum_{k=0}^{\infty}
	\left(
	\int^{(k+1)h}_{kh}
	x(t)^{\top} Q_c x(t) + u(t)^{\top} R_c u (t) 
	dt
	\right) \\
	&=
	\sum_{k=0}^{\infty}
	E\left(
	\int^{(k+1)h}_{kh}
	x(t)^{\top} Q_c x(t) + u(t)^{\top} R_c u (t) 
	dt
	\right).
	\end{align*}
	It also follows from Lemma \ref{lem:reduction_DT_LQ_cost} that
	\begin{align*}
	&\sum_{k=0}^{\infty}
	E\left(
	\int^{(k+1)h}_{kh}
	x(t)^{\top} Q_c x(t) + u(t)^{\top} R_c u (t) 
	dt
	\right)\\
	&\qquad  = 
	\sum_{k=0}^{\infty}
	E\left(
	\begin{bmatrix}
	\xi_k \\ u_k
	\end{bmatrix}^{\top}
	\begin{bmatrix}
	Q(\phi_k) & W(\phi_k) \\ \star & R(\phi_k)
	\end{bmatrix}
	\begin{bmatrix}
	\xi_k \\ u_k
	\end{bmatrix}
	\right) = \mathcal{J}_{d1}(\widehat \mu,u).
	\end{align*}
	Thus if a control input 
	$u^{\rm{opt}}  \in \mathcal{U}_c$ is a solution to Problem \ref{prob:LQ}, then $u^{\rm{opt}} $ satisfies
	$u^{\rm{opt}} \in \mathcal{U}_d$ and is a solution to
	Problem \ref{prob:LQ_dis} 
	where the system matrices $A$, $B$ and 
	the weighting matrices $Q$, $W$, $R$ are defined by \eqref{eq:A_B_def}, \eqref{eq:QWR_def}, 
	and vice versa.
\end{proof}

Let us next remove the cross term of the cost function $\mathcal{J}_{d1}$.
To this end, as in the deterministic case \cite[Section~3.4]{Anderson1990}, we transform 
$u_k$ into $\bar u_k$
in the following way:
\begin{equation}
\label{eq:input_transformation}
\bar u_k = u_k + R(\phi_k)^{-1}W(\phi_k)^{\top} \xi_k\qquad \forall k \in \mathbb{Z}_+.
\end{equation}
Since $Q_c \geq 0$, $R_c > 0$, and $h -\tau_{\max} > 0$, 
it follows that $R(\phi)$ in \eqref{eq:QWR_def} 
is invertible for all $\phi \in \mathcal{M}$. Therefore
the right-hand side of \eqref{eq:input_transformation} is well-defined for all $\phi_k \in \mathcal{M}$.

Define 
\begin{equation}
\label{eq:tilde_A_def}
\bar A(\phi) := A(\phi) - B(\phi)R(\phi)^{-1}W(\phi)^{\top} \qquad \forall \phi\in \mathcal{M}
\end{equation}
and let $C(\phi)$ and $D(\phi)$ be the matrices obtained from the following Cholesky
decompositions:
\begin{subequations}
	\label{eq:Cholesky_decomposition}
	\begin{align}
	C(\phi)^{\top}C(\phi) &= Q(\phi) - W(\phi) R(\phi)^{-1} W(\phi)^{\top} 
	\qquad \forall \phi \in \mathcal{M} \\
	D(\phi)^{\top}D(\phi) &= R(\phi) 
	\qquad \forall \phi \in \mathcal{M}.
	\end{align}
\end{subequations}
These Cholesky decompositions are possible if 
the weighting matrix in Lemma~\ref{lem:reduction_DT_LQ_cost}
satisfies 
\begin{equation}
\label{eq:weighting_matrix_SPD}
\begin{bmatrix}
Q(\phi) & W(\phi) \\ \star & R(\phi)
\end{bmatrix} \geq 0
\qquad  \forall \phi \in \mathcal{M}.
\end{equation}
This is because $R(\phi) > 0$ for every  $\phi \in \mathcal{M}$ and 
the Schur complement formula leads to
\[
Q(\phi) - W(\phi) R(\phi)^{-1} W(\phi)^{\top}  \geq 0\qquad  \forall \phi \in \mathcal{M}.
\]
Under the transformation \eqref{eq:input_transformation}, we obtain the following result:
\begin{lemma}
	\label{lem:reduction2}
	Assume that the weighting matrices $Q$, $W$, $R$ in \eqref{eq:QWR_def}
	satisfy the inequality \eqref{eq:weighting_matrix_SPD}.
	A control input $u^{\rm{opt}}$ is a solution to
	Problem \ref{prob:LQ_dis} with the system \eqref{eq:discretized_sys}
	and the LQ cost \eqref{eq:Jd1} if and only if
	$\bar u^{\rm{opt}}$ with $\bar u^{\rm{opt}}_k = u^{\rm{opt}}_k + R(\phi_k)^{-1}W(\phi_k)^{\top} \xi_k$
	is a solution to Problem \ref{prob:LQ_dis} where the Markov jump system is given by
	\begin{equation}
	\label{eq:state_eq_til_u}
	\xi_{k+1} =  \bar A(\phi_k)
	\xi_k + B(\phi_k) \bar u_k
	\end{equation}
	and the LQ cost by
	\begin{equation}
	\label{eq:Jc_til_u}
	\mathcal{J}_{d2}(\widehat \mu, \bar u) := 
	\sum_{k=0}^{\infty} E\left(
	\|C(\phi_k) \xi_k\|^2 + \|D(\phi_k) \bar u_k\|^2
	\right).
	\end{equation}
	Here the matrices $\bar A$, $C$, and $D$ are defined as in 
	\eqref{eq:tilde_A_def} and \eqref{eq:Cholesky_decomposition}.
\end{lemma}
\begin{proof}
	Let $\xi$ and $\bar \xi$ be the solutions of the difference equations 
	\eqref{eq:discretized_sys} and \eqref{eq:state_eq_til_u} with the same Markov parameter $\phi$.
	If $\bar u_k$ satisfies \eqref{eq:input_transformation} and if $\xi_k = \bar \xi_k$, then
	\begin{align*}
	\xi_{k+1} &= A(\phi_k)\xi_k + B(\phi_k) u_k  \\
	&=A(\phi_k) \xi_k + B(\phi_k) (\bar u_k - R(\phi_k)^{-1}W(\phi_k)^{\top} \xi_k) \\
	&=   \bar A(\phi_k)
	\bar \xi_k + B(\phi_k) \bar u_k = \bar \xi_{k+1}.
	\end{align*}
	Therefore,  if $\xi_0 = \bar \xi_0$, then
	$
	\xi_k = \bar \xi_k
	$
	for every $ k \in \mathbb{Z}_+$.
	Thus $u \in \mathcal{U}_d$ for the system \eqref{eq:discretized_sys}
	if and only if $\bar u \in \mathcal{U}_d$ for the system \eqref{eq:state_eq_til_u}.	
	Moreover, if \eqref{eq:input_transformation} holds, then
	\begin{align*}
	\mathcal{J}_{d2}(\widehat \mu, \bar u) 
	&= 
	\sum_{k=0}^{\infty} E\left(
	\begin{bmatrix}
	\bar \xi_k \\ \bar u_k
	\end{bmatrix}^{\top}
	\begin{bmatrix}
	C^{\top}(\phi_k)C(\phi_k) & 0\\ \star & D^{\top}(\phi_k)D(\phi_k) 
	\end{bmatrix}
	\begin{bmatrix}
	\bar \xi_k \\ \bar u_k
	\end{bmatrix}
	\right) \\
	&=
	\sum_{k=0}^{\infty} E\left(
	\begin{bmatrix}
	\bar \xi_k \\ \bar u_k
	\end{bmatrix}^{\top}
	\begin{bmatrix}
	Q(\phi_k) - W(\phi_k) R(\phi_k)^{-1} W(\phi_k)^{\top}  & 0\\ \star & R(\phi_k) 
	\end{bmatrix}
	\begin{bmatrix}
	\bar \xi_k \\ \bar u_k
	\end{bmatrix}
	\right) \\
	&=
	\sum_{k=0}^{\infty} E\left(
	\begin{bmatrix}
	\xi_k \\ u_k
	\end{bmatrix}^{\top}
	\begin{bmatrix}
	Q(\phi_k)& W(\phi_k) \\ \star & R(\phi_k) 
	\end{bmatrix}
	\begin{bmatrix}
	\xi_k \\ u_k
	\end{bmatrix}
	\right) = 	\mathcal{J}_{d1}(\widehat \mu, u). 
	\end{align*}
	Hence 
	$u^{\rm{opt}}$ is a solution to
	Problem \ref{prob:LQ_dis} with the system \eqref{eq:discretized_sys}
	and the LQ cost \eqref{eq:Jd1} if and only if
	$\bar u^{\rm{opt}}$ with $\bar u^{\rm{opt}}_k = u^{\rm{opt}}_k + R(\phi_k)^{-1}W(\phi_k)^{\top} \xi_k$
	is a solution to Problem \ref{prob:LQ_dis} with the system
	\eqref{eq:state_eq_til_u} and
	the LQ cost \eqref{eq:Jc_til_u}.
\end{proof}

Finally
we can reduce the LQ problem \ref{prob:LQ}
for sampled-data systems with stochastic delays
to the LQ problem \ref{prob:LQ_dis}  for discrete-time 
Markov jump systems and LQ costs in the form \eqref{eq:Jc_til_u}.
\begin{theorem}
	\label{thm:reduction}
	A control input 
	$u^{\rm{opt}} $ is a solution to Problem \ref{prob:LQ} if and only if
	$\bar u^{\rm{opt}} $ with $\bar u^{\rm{opt}}_k = u^{\rm{opt}}_k + R(\phi_k)^{-1}W(\phi_k)^{\top} \xi_k$ is a solution to
	Problem \ref{prob:LQ_dis} for the Markov jump system \eqref{eq:state_eq_til_u}
	and the LQ cost \eqref{eq:Jc_til_u}, where the matrices $\bar A$, $C$, and $D$ are obtained in
	\eqref{eq:tilde_A_def} and \eqref{eq:Cholesky_decomposition}
	from the matrices $A$, $B$, $Q$, $W$, and $R$ defined by  \eqref{eq:A_B_def} and \eqref{eq:QWR_def}.
\end{theorem}
\begin{proof}
	This is an immediate consequence of
	Lemmas \ref{lem:reduction1} and \ref{lem:reduction2}.
\end{proof}

\begin{remark}
	\label{rem:uniqueness_CD}
	{ \em
		If, in \eqref{eq:weighting_matrix_SPD}, we have strict positive definiteness
		\begin{equation}
		\label{eq:weighting_matrix}
		\begin{bmatrix}
		Q(\phi) & W(\phi) \\ \star & R(\phi)
		\end{bmatrix} > 0
		\qquad  \forall \phi \in \mathcal{M}
		\end{equation}
		instead of the semidefiniteness, then
		\[
		Q(\phi) - W(\phi) R(\phi)^{-1} W(\phi)^{\top} > 0
		\qquad  \forall \phi \in \mathcal{M}
		\]
		by the Schur complement formula. 
		In this case, $C(\phi)$ and $D(\phi)$,
		derived from
		the Cholesky decompositions \eqref{eq:Cholesky_decomposition},
		are unique in the following sense:
		For all $\phi \in \mathcal{M}$,
		there exist unique upper triangular matrices
		$C(\phi)$ and $D(\phi)$ with strictly 
		positive diagonal entries such that \eqref{eq:Cholesky_decomposition} holds.
		Moreover, 
		$C(\phi)$, $D(\phi)$
		are continuous with respect to $\phi$.  See, e.g., Chapters 9 and 12 of \cite{Schatzman2002}.
		Thus $C$ and $D$ satisfies $C \in \mathbb{H}^{n \times n}_{\sup} $ and $D \in \mathbb{H}^{m \times m}_{\sup} $.
	}
\end{remark}


\section{LQ control for Discrete-time Markov Jump Systems}
In the previous section,
we reduced the LQ problem \ref{prob:LQ}
for sampled-data systems with stochastic delays
into the LQ problem \ref{prob:LQ_dis}  for discrete-time 
Markov jump systems and LQ costs in the form \eqref{eq:Jc_til_u}.
In this section, we recall results from \cite{Costa2015} 
on such an LQ problem for Markov jump systems.

First we define stochastic stability 
for discrete-time Markov jump linear systems. 
On a probability space $(\Omega, \mathcal{F}, P)$,
consider the following autonomous system
\begin{equation}
\label{eq:automous_sys}
\xi_{k+1} = A(\phi_k)\xi_k,
\end{equation} 
where $A \in {\mathbb H}^{n}_{\sup}$ and 
the sequence $\{ \phi_k:k \in \mathbb{Z}_+\}$ is
a time-homogeneous Markov chain in a Borel space $\mathcal{M}$.
Throughout this section, we assume that
the initial distribution $\widehat \mu$
of $(\xi_0,\phi_0)$, which is a probability measure on $\mathbb{R}^n \times \mathcal{M}$,
satisfies the following conditions analogous to the ones in Assumption \ref{assump:initial_dist}:
A1')
$E(\|\xi_0\|^2) < \infty$ and A2')
$\widehat \mu_{\mathcal{M}}(\bullet) = \widehat \mu(\mathbb{R}^n \times \bullet)$ 
is absolutely continuous with respect to $\mu$.

\begin{definition}[Stochastic stability, \cite{Costa2014}]
	The autonomous Markov jump linear system \eqref{eq:automous_sys}
	is said to be stochastically stable if
	$\sum_{k=0}^{\infty} E(\|\xi_k\|^2) < \infty$ for any initial distribution $\widehat \mu$
	satisfying A1') and A2').
\end{definition}

Let $g(\bullet,\bullet)$ be the density function 
with respect to a $\sigma$-finite measure $\mu$
on $\mathcal{M}$ 
for the transition of the  Markov chain 
$\{\phi_k:k \in \mathbb{Z}_+ \}$ as in Assumption \ref{assump:delay}.
For every $A \in {\mathbb H}^n_{\sup}$,
define an operator $\mathcal{L}_A: \mathbb{H}^n_{1,\mathbb{C}} \to  \mathbb{H}^n_{1,\mathbb{C}} $ by
\begin{align}
\label{eq:L_def}
\mathcal{L}_{A}(V)(\bullet) &:= 
\int_{\mathcal{M}} g(\ell, \bullet) A(\ell) V(\ell) A(\ell) ^{\top} \mu(d\ell).
\end{align}

We recall a relationship among stochastic stability, 
the spectral radius $r_{\sigma}(\mathcal{L}_{A})$, and
a Lyapunov inequality condition.
\begin{theorem}[\hspace{-3.5pt} \cite{Costa2014}]
	\label{thm:iff_conds}
	The following assertions are equivalent:
	\begin{enumerate}
		\item The system \eqref{eq:automous_sys} is stochastically stable.
		\item The spectral radius $r_{\sigma}(\mathcal{L}_A) < 1$, where
		$\mathcal{L}_A$ is defined as in \eqref{eq:L_def}.
		\item
		There exist $S \in \mathbb{H}^{n+}_{\sup,\mathbb{C}}$ and $\epsilon >0$
		such that the following Lyapunov inequality 
		holds for $\mu$-almost every $\phi \in \mathcal{M}$:
	\end{enumerate}
	\begin{equation}
	\label{eq:Lyap_type_iff}
	S(\phi) - 
	A(\phi)^{\top} 
	\left(
	\int_{\mathcal{M}} g(\phi,\ell) S(\ell)\mu(d\ell)
	\right)
	A(\phi) \geq \epsilon I.
	\end{equation}
\end{theorem}
Note that the matrix $S$ in the statement 3 is  a complex-valued function,
because the system matrix $A$ is a complex-valued function in \cite{Costa2014}.
However,
for a real-valued function $A$,
it is enough to find a real-valued $S$ in the statement 3.
\begin{proposition}
	\label{prop:real_case}
	For a real-valued function $A \in \mathbb{H}^{n}_{\sup}$,
	the statement 3 in Theorem \ref{thm:iff_conds} is equivalent to
	\begin{enumerate}
		\renewcommand{\labelenumi}{\arabic{enumi}'.}
		\setcounter{enumi}{2}
		{\em
			\item {\em There exist $S \in \mathbb{H}^{n+}_{\sup}$ and $\epsilon >0$
				such that the Lyapunov inequality \eqref{eq:Lyap_type_iff}
				holds for $\mu$-almost every $\phi \in \mathcal{M}$.
			}
		}
	\end{enumerate}
\end{proposition}
\begin{proof}
	The statement 3' $\Rightarrow$ 3 is trivial because 
	$\mathbb{H}^{n+}_{\sup} \subset \mathbb{H}^{n+}_{\sup,\mathbb{C}}$. To prove 3 $\Rightarrow$ 3',
	we let $S \in \mathbb{H}^{n+}_{\sup,\mathbb{C}}$ and $\epsilon >0$ satisfy the Lyapunov inequality \eqref{eq:Lyap_type_iff}.
	Let $S_R(\phi) \in \mathbb{R}^{n \times n}$ and $S_I(\phi)\in \mathbb{R}^{n \times n}$ 
	be the real and imaginary part of $S(\phi)$, that is, 
	\[
	S(\phi) = S_R(\phi) + i S_I(\phi).
	\]
	Since $S(\phi)$ is Hermitian, it follows that $S_R(\phi) = S_R(\phi) ^{\top}$ 
	and $S_I(\phi) = -S_I(\phi)^{\top}$.
	Therefore, 
	\[
	0 \leq x^{\top}S(\phi)x = x^{\top}(S_R(\phi) + i S_I(\phi)) x = x^{\top}S_R(\phi) x
	\qquad \forall x \in \mathbb{R}^n.
	\]
	Thus we obtain $S_R \in \mathbb{H}^{n+}_{\sup}$.
	Similarly, since $A(\phi)$ and $g(\phi,\ell)$ are real-valued, it follows that 
	\begin{align*}
	\epsilon \|x\|^2 
	&\leq 
	x^{\top}
	\left(
	S(\phi) - 
	A(\phi)^{\top} 
	\left(
	\int_{\mathcal{M}} g(\phi,\ell) S(\ell)\mu(d\ell)
	\right)
	A(\phi)	
	\right)
	x \\
	&=
	x^{\top}
	\left(
	S_R(\phi) - 
	A(\phi)^{\top} 
	\left(
	\int_{\mathcal{M}} g(\phi,\ell) S_R(\ell)\mu(d\ell)
	\right)
	A(\phi)	
	\right)
	x\qquad \forall x \in \mathbb{R}^n.
	\end{align*}
	Hence $S_R$ also satisfies the Lyapunov inequality \eqref{eq:Lyap_type_iff}
	for $\mu$-almost every $\phi \in \mathcal{M}$.
	This completes the proof.
\end{proof}

We next provide the definition of stochastic stabilizability and
stochastic detectability.
\begin{definition}[Stochastical stabilizability, \cite{Costa2015}]
	\label{def:SS_def}
	Let $A \in {\mathbb H}^{n}_{\sup}$ and
	$B \in {\mathbb H}^{n \times m}_{\sup}$.
	We say that $(A,B)$ is stochastically stabilizable if 
	there exists $F \in {\mathbb H}^{m \times n}_{\sup}$ 
	such that $r_{\sigma}(\mathcal{L}_{A+BF}) < 1$, where
	$\mathcal{L}_{A+BF}$ is defined as in \eqref{eq:L_def}.
	In this case, $F$ is said to stochastically stabilize $(A,B)$.
\end{definition}

\begin{definition}[Stochastic detectability, \cite{Costa2015}]
	\label{def:SD_def}
	Let $A \in {\mathbb H}^{n}_{\sup}$
	and $C \in {\mathbb H}^{r,n}_{\sup}$.
	We say that $(C,A)$ is stochastically detectable if 
	there exists $L \in {\mathbb H}^{n \times r}_{\sup}$ 
	such that \mbox{$r_{\sigma}(\mathcal{L}_{A+LC}) < 1$}, 
	where
	$\mathcal{L}_{A+LC}$ is defined as in \eqref{eq:L_def}.
\end{definition}

Consider the controlled system 
\begin{equation*}
\xi_{k+1} =  A(\phi_k)
\xi_k + B(\phi_k) u_k
\end{equation*}
and the
LQ cost 
\begin{equation*}
\mathcal{J}_d(\widehat \mu, u) = 
\sum_{k=0}^{\infty} E\left(
\|C(\phi_k) \xi_k\|^2 + \|D(\phi_k) u_k\|^2
\right).
\end{equation*}
where
$A \in {\mathbb H}^{n}_{\sup}$, $B \in {\mathbb H}^{m \times n}_{\sup}$,
$C \in {\mathbb H}^{n \times r}_{\sup}$, and $D \in {\mathbb H}^{m \times q}_{\sup}$.
We assume that there exists $\epsilon_D >0$ such that 
$D(\phi)^{\top}D(\phi) > \epsilon_DI$ for $\mu$-almost every $\phi \in \mathcal{M}$.
As in Section 3,
let $\{\mathcal{F}^d_k:k \in \mathbb{Z}_+\}$ denote a filtration, where
$\mathcal{F}^d_k$ represents the $\sigma$-field generated by
$
\{\xi_0,\phi_0,\dots,\xi_k,\phi_k \}
$, and
set $\mathcal{U}_d$ as the class of
control inputs $u = \{u_k: k \in \mathbb{Z}_+\}$ such that $u_k$ is $\mathcal{F}^d_k$
measurable and 
$E(\|\xi_k\|^2) \to 0$ as $k\to \infty$ for
every initial distribution $\widehat \mu$ satisfying A1') and A2').

Define operators 
$\mathcal{E}:{\mathbb H}^{n+}_{\sup} \to {\mathbb H}^{n+}_{\sup}$, 
$\mathcal{V}:{\mathbb H}^{n+}_{\sup} \to {\mathbb H}^{m+}_{\sup}$, and
$\mathcal{R}:{\mathbb H}^{n+}_{\sup} \to {\mathbb H}^{n+}_{\sup}$ as follows:
\begin{align*}
\mathcal{E}(Z)(\bullet) &:=
\int_{\mathcal{M}} Z(\ell)g(\bullet,\ell) \mu(d\ell)\\
\mathcal{V}(Z) &:= 
D^{\top} D+ B^{\top}\mathcal{E}(Z)B \\
\mathcal{R}(Z) &:=
C^{\top} C+
A^{\top}(\mathcal{E}(Z) - \mathcal{E}(Z)B \mathcal{V}(Z)^{-1}B^{\top}
\mathcal{E}(Z)) A.
\end{align*}
Using these operators, 
we can obtain a solution to the LQ problem for discrete-time Markov jump linear systems from
the iterative computation of $\mathcal{R}(Z)$.
\begin{theorem}[\hspace{-3.6pt} \cite{Costa2015}]
	\label{thm:LQ_problem_MJS}
	Consider the Markov jump system \eqref{eq:state_eq_til_u} with the
	LQ cost $\mathcal{J}_d$ in \eqref{eq:Jc_til_u}.
	If $( A,B)$ is stochastically stabilizable and $(C,A)$ is
	stochastically detectable, then
	there exists a  function $S \in {\mathbb H}^{n+}_{\sup}$ such that
	$S$ is the unique solution in  ${\mathbb H}^{n+}_{\sup}$ of the $\mathcal{M}$-coupled algebraic Riccati equation
	\begin{equation}
	\label{eq:MCARE}
	S(\phi) = \mathcal{R}(S)(\phi)\quad \text{$\mu$-almost every $\phi \in \mathcal{M}$}
	\end{equation}
	and such that
	\[
	K:= -\mathcal{V}(S)^{-1} B^{\top}\mathcal{E}(S) A \in {\mathbb H}^{m \times n}_{\sup}
	\] 
	stochastically stabilizes $(A,B)$.
	The control input $u^{\rm{opt}}\in \mathcal{U}_d$ defined by
	$u_k^{\rm{opt}} := K(\phi_k)\xi_k$
	achieves \[\mathcal{J}_d(\widehat \mu,u^{\rm{opt}}) 
	= \inf_{u \in \mathcal{U}_d} \mathcal{J}_d(\widehat \mu,u)
	= E(\xi_0^{\top}S(\phi_0)\xi_0)\] for every initial distribution $\widehat \mu$
	satisfying A1') and A2').
	Moreover, we can compute the solution $S$ of the Riccati equation \eqref{eq:MCARE} in
	the following way:
	For any $\Xi \in \mathbb{H}^{n+}_{\sup}$, the sequence $\{Y_k^{\eta}\}_{k=0}^{\eta}$ that is
	calculated 
	by solving a (backward recursive) Riccati difference equation 
	$Y_k^{\eta} = \mathcal{R}(Y^{\eta}_{k+1})$ with the initial value $Y^{\eta}_\eta =\Xi$
	satisfies
	$Y^\eta_0(\phi) \to S(\phi)$ as $\eta \to \infty$ for $\mu$-almost every $\phi \in \mathcal{M}$.
\end{theorem}

Let us go back to the reduced LQ problem \ref{prob:LQ_dis} 
for the Markov jump system \eqref{eq:state_eq_til_u}
and the LQ cost \eqref{eq:Jc_til_u}, where the matrices $\bar A$, $C$, and $D$ are defined by
\eqref{eq:tilde_A_def} and \eqref{eq:Cholesky_decomposition}.
The Cholesky decompositions for all $\phi \in \mathcal{M}$ 
in \eqref{eq:Cholesky_decomposition} requires 
heavy computational cost, but
the weighting functions $C$ and $D$ appear only in $\mathcal{R}$
and $\mathcal{V}$ in the forms $C^{\top}C$ and
$D^{\top}D$.
Hence, to compute an optimal control input $u^{\text{opt}}$,
we do not need the 
Cholesky decompositions in \eqref{eq:Cholesky_decomposition}.
Although we still need $C$ to check the stochastic detectability of $(C,\bar A)$,
we see from Proposition \ref{prop:sssd} below that
it is enough to test the stochastic detectability of $(C^{\top}C,\bar A)$
if $C^{\top}C$ is positive definite.
\begin{proposition}
	\label{prop:sssd}
	Define $\bar A$ and $C$ as in \eqref{eq:tilde_A_def}
	and \eqref{eq:Cholesky_decomposition}, respectively.
	The pair $(A,B)$ is stochastically stabilizable if and only if 
	$(\bar A,B)$ is stochastically stabilizable.
	Moreover, under the positive definiteness of
	the weighting matrix in \eqref{eq:weighting_matrix},
	$(C,\bar A)$ is stochastically detectable if and only if $(C^{\top}C,\bar A)$ is stochastically detectable.
\end{proposition}
\begin{proof}
	By the definition of $\bar A$, 
	$K \in {\mathbb H}^{m \times n}_{\sup}$ stochastically stabilizes
	$(A,B)$ if and only if $K -R^{-1}W \in {\mathbb H}^{m \times n}_{\sup}$ 
	stochastically stabilizes $(\bar A,B)$. 
	
	From the discussion in Remark \ref{rem:uniqueness_CD}, $C^{-1} \in {\mathbb H}^{n \times n}_{\sup}$
	if \eqref{eq:weighting_matrix} holds.
	Hence 
	if $(C,\bar A)$ is stochastically detectable with an
	observer gain $L \in {\mathbb H}^{n \times n}_{\sup}$,
	then $(C^{\top}C,\bar A)$  is also stochastically detectable with an
	observer gain $L(C^{\top})^{-1}\in {\mathbb H}^{n \times n}_{\sup}$, and vice versa.
\end{proof}

\section{Sufficient conditions for Stochastic Stabilizability and 
	Detectability}
From the results in Sections~3 and 4, we can obtain an optimal controller
under the assumption that 
$(A, B)$ defined by \eqref{eq:A_B_def} is
stochastically stabilizable and
$(C,\bar A)$
defined by
\eqref{eq:tilde_A_def}
and \eqref{eq:Cholesky_decomposition}
is stochastically
detectable.
This assumption does not hold in general (and hence the solution of
the Riccati difference equation may diverge)
even if $(A_c, B_c)$
is controllable and $(Q_c, A_c)$ is observable. 
The major difficulty in this controller design
is to check the stochastic stabilizability and detectability, namely, 
to show the existence of 
$F \in {\mathbb H}^{m \times n}_{\sup}$ and 
$L \in {\mathbb H}^{n \times r}_{\sup}$  such that 
the spectral radii of the operators
$\mathcal{L}_{A+BF}$ and  
$\mathcal{L}_{\bar A+LC}$ are less than one.
In this section, we provide sufficient conditions
for these properties
in terms of LMIs.
To this end,
we use the following technical result:
\begin{lemma}[\hspace{-3.5pt} \cite{Geromel2007}]
	\label{lem:geromel_LMI}
	For every
	square matrix $U$ and every positive definite matrix 
	$S$,
	\[
	US^{-1}U^{\top} \geq U + U^{\top} - S.
	\]
\end{lemma}

We here assume that $\{ \phi_k:k \in \mathbb{Z}_+\}$ is a time-homogeneous
Markov chain taking values in 
the box $\mathcal{M} = [\tau_{\min}, \tau_{\max}]^p$.
As in Assumption \ref{assump:delay},
let $g(\bullet,\bullet)$ be the density function 
with respect to a $\sigma$-finite measure $\mu$
on $\mathcal{M}$ 
for the transition of the Markov chain 
$\{\phi_k:k \in \mathbb{Z}_+ \}$.

\subsection{Stochastic Stabilizability}
We first study the stochastic stabilizability of the 
pair $A \in {\mathbb H}^{n}_{\sup}$ and
$B \in {\mathbb H}^{n\times m}_{\sup}$.
In our LQ problem, we need to check the stochastic stabilizability of the pair
$(A,B)$ defined by \eqref{eq:A_B_def}.

Divide $\mathcal{M} =  [\tau_{\min}, \tau_{\max}]^p$ into $N$ disjoint boxes $\{\mathcal{B}_i\}_{i=1}^{N}$
(whose union is $\mathcal{M}$),
e.g., by splitting each interval $[\tau_{\min}, \tau_{\max}]$ into
$r$ intervals $[s_i,s_{i+1})$ ($i=1,\dots,r-1$) and $[s_r,s_{r+1}]$ such that
\begin{equation*}
\tau_{\min} =: s_1 < s_2 < \cdots < s_{r+1} := \tau_{\max}.
\end{equation*}
For each $i=1,\dots,N$,
let $c_i \in \mathcal{B}_i$, e.g., the center of $\mathcal{B}_i$.
Consider a
piecewise-constant feedback gain $F \in {\mathbb H}^{m \times n}_{\sup}$ 
defined by
\begin{equation}
\label{eq:controller}
F(\phi) := F_i \in \mathbb{R}^{m\times n} \qquad \forall \phi \in \mathcal{B}_i.
\end{equation}

We provide a sufficient condition 
for the feedback gain $F$ in \eqref{eq:controller}
to stochastically stabilize $(A,B)$, inspired by
the gridding approach developed, e.g., in  
\cite{Fujioka2009,Oishi2010, Hetel2011, Donkers2011,Donkers2012_stochastic}.
\begin{theorem}
	\label{thm:design_of_gain}
	Let $A \in {\mathbb H}^{n}_{\sup}$ and
	$B \in {\mathbb H}^{n \times m}_{\sup}$.
	For each $i=1,\dots,N$, 
	define
	\begin{align}
	\label{eq:beta_def}
	w_i(\phi) 
	&:= \int_{\mathcal{B}_i} g(\phi,\ell) \mu(d\ell)  \geq 0\qquad  \forall \phi \in \mathcal{M},
	\end{align}
	and
	\begin{equation*}
	\begin{bmatrix}
	\Gamma_{A,i} & \Gamma_{B,i}
	\end{bmatrix} := 
	\begin{bmatrix}
	\sqrt{w_1(c_i)} A(c_i) &\sqrt{w_1(c_i)} B(c_i) \\
	\vdots & \vdots \\
	\sqrt{w_N(c_i)} A(c_i) &
	\sqrt{w_N(c_i)} B(c_i)
	\end{bmatrix}.
	\end{equation*}
	Assume that,
	for every $i=1,\dots,N$, 
	a scalar $\kappa_i > 0$ satisfies
	\begin{align}
	\label{eq:kappa_cond}
	\left\|
	\begin{bmatrix}
	\sqrt{w_1(\phi)} A(\phi) &\sqrt{w_1(\phi)} B(\phi) \\
	\vdots & \vdots \\
	\sqrt{w_N(\phi)} A(\phi) &
	\sqrt{w_N(\phi)} B(\phi)
	\end{bmatrix}
	-
	\begin{bmatrix}
	\Gamma_{A,i} &  \Gamma_{B,i}
	\end{bmatrix}
	\right\| 
	\leq \kappa_i
	\end{align}
	for $\mu$-almost every $\phi \in \mathcal{B}_i$.
	If
	there exist positive definite matrices $R_i \in \mathbb{R}^{n \times n}$, 
	(not necessarily symmetric) matrices $U_i \in \mathbb{R}^{n \times n}, \bar F_i\in \mathbb{R}^{m \times n}$, and
	scalars $\lambda_i > 0$ such that
	the following LMIs hold for all $i=1\dots,N$:
	\begin{equation}
	\label{eq:LMI_feedback_design}
	\begin{bmatrix}
	U_i   +   U_i^{\top}   -   R_i   & 0    
	& U_i^{\top}\Gamma_{A,i}^{\top}   +   \bar F_i^{\top} \Gamma_{B,i}^{\top}   & 
	\kappa_i 
	\begin{bmatrix}
	U_i^{\top}      & \bar F_i^{\top}
	\end{bmatrix}\\
	\star   & \lambda_i I   & \lambda_i I   & 0 \\
	\star   & \star   & \bm{R}   & 0 \\
	\star   & \star   & \star   & \lambda_i I
	\end{bmatrix}   >   0,
	\end{equation}
	where $\bm{R} := \diag(R_1,\dots,R_{N})$,
	then
	the pair $(A,B)$ is stochastically stabilizable by
	the controller \eqref{eq:controller} with $F_i := \bar F_i U_i^{-1}$.
\end{theorem}

\begin{proof}
	From Theorem \ref{thm:iff_conds} and Proposition \ref{prop:real_case},
	$(A,B)$ is stochastically stabilizable if and only if
	there exist 
	$S \in \mathbb{H}^{n+}_{\sup}$, $F \in \mathbb{H}^{m\times n}_{\sup}$, and $\epsilon >0$
	such that the following Lyapunov inequality 
	holds for $\mu$-almost every $\phi \in \mathcal{M}$:
	\begin{equation}
	\label{eq:Lyap_type_SS}
	S(\phi) - 
	(A(\phi)+B(\phi)F(\phi))^{\top} 
	\left(
	\int_{\mathcal{M}} g(\phi,\ell) S(\ell)\mu(d\ell)
	\right)
	(A(\phi)+B(\phi)F(\phi)) \geq \epsilon I.
	\end{equation}
	We employ a piecewise-constant matrix function $S$ for
	the Lyapunov inequality \eqref{eq:Lyap_type_SS}.
	Define 
	\begin{equation}
	\label{eq:S_approximation}
	S(\phi) := S_i\qquad \forall \phi \in \mathcal{B}_i
	\end{equation}
	with $S_i > 0$. 
	In what follow, we prove that 
	if the LMIs in \eqref{eq:LMI_feedback_design} are feasible for all $i=1,\dots,N$,
	then the Lyapunov inequality \eqref{eq:Lyap_type_SS} holds with
	$S\in {\mathbb H}^{n+}_{\sup}$ defined by \eqref{eq:S_approximation}.
	
	By construction, we have that 
	\begin{equation}
	\label{eq:integral_to_summation}
	\int_{\mathcal{M}} g(\phi,\ell) S(\ell)\mu(d\ell)
	=
	\sum_{i=1}^{N}
	w_i(\phi) S_i \qquad \forall \phi \in \mathcal{M}.
	\end{equation}
	Note that
	\begin{equation}
	\label{eq:summation_to_quadratic}
	\sum_{i=1}^{N}
	w_i(\phi) S_i = 
	\begin{bmatrix}
	\sqrt{w_1(\phi)} I \\ \vdots \\ \sqrt{w_{N}(\phi)} I
	\end{bmatrix}^{\top} \bm{S}
	\begin{bmatrix}
	\sqrt{w_1(\phi)} I \\ \vdots \\  \sqrt{w_{N}(\phi)} I
	\end{bmatrix},
	\end{equation}
	where $\bm{S} := \diag(S_1,\dots,S_{N})$.
	Substituting \eqref{eq:integral_to_summation} and \eqref{eq:summation_to_quadratic} 
	into \eqref{eq:Lyap_type_SS}, we see that 
	\eqref{eq:Lyap_type_SS} can be transformed into
	the matrix inequality
	\begin{equation}
	\label{eq:SS_matrix_inequality_before_SC}
	S(\phi) -
	(A(\phi) + B(\phi)F(\phi))^{\top}
	\begin{bmatrix}
	\sqrt{w_1(\phi)} I \\ \vdots \\ \sqrt{w_N(\phi)} I
	\end{bmatrix}^{\top} \bm{S}
	\begin{bmatrix}
	\sqrt{w_1(\phi)} I \\ \vdots \\  \sqrt{w_N(\phi)} I
	\end{bmatrix} 
	(A(\phi) + B(\phi)F(\phi)) 
	\geq \epsilon I.
	\end{equation}
	Moreover,
	by the Schur complement formula,
	the matrix inequality \eqref{eq:SS_matrix_inequality_before_SC} is
	equivalent to
	\begin{equation}
	\label{eq:beta_LMI}
	\begin{bmatrix}
	S(\phi) - \epsilon I & (A(\phi) + B(\phi)F(\phi))^{\top}
	\begin{bmatrix}
	\sqrt{w_1(\phi)} I \\ \vdots \\ \sqrt{w_N(\phi)} I
	\end{bmatrix}^{\top}
	\bm{S} \\
	\star & \bm{S}
	\end{bmatrix} \geq 0.
	\end{equation}
	
	Using the inequality \eqref{eq:kappa_cond}, 
	we next discretize \eqref{eq:beta_LMI}, more specifically,
	show that if the LMIs \eqref{eq:LMI_feedback_design} are feasible
	for every $i=1,\dots,N$, then the matrix inequality 
	\eqref{eq:beta_LMI} holds for $\mu$-almost every $\phi \in \mathcal{M}$.
	In terms of the upper-right part of the matrix in \eqref{eq:beta_LMI}, we obtain
	\begin{align*}
	&(A(\phi) + B(\phi)F(\phi))^{\top}\begin{bmatrix}
	\sqrt{w_1(\phi)} I \\ \vdots \\ \sqrt{w_N(\phi)} I
	\end{bmatrix}^{\top} \\
	&\qquad = 
	\begin{bmatrix}
	I & F(\phi)^{\top}
	\end{bmatrix}
	\begin{bmatrix}
	\sqrt{w_1(\phi)} A(\phi)^{\top} & \cdots & \sqrt{w_N(\phi)} A(\phi)^{\top} \\[2pt]
	\sqrt{w_1(\phi)} B(\phi)^{\top} & \cdots & \sqrt{w_N(\phi)} B(\phi)^{\top}
	\end{bmatrix}.
	\end{align*}
	We also have from the inequality \eqref{eq:kappa_cond} that,
	for $\mu$-almost every $\phi \in \mathcal{B}_i$, there exists
	\[
	\begin{bmatrix}
	\Phi_A & \Phi_B
	\end{bmatrix} \in \mathbb{R}^{nN \times( n+m)}
	\text{~with~}
	\left\|
	\begin{bmatrix}
	\Phi_A & \Phi_B
	\end{bmatrix} 
	\right\| < 1\] 
	such that
	\begin{align*}
	\begin{bmatrix}
	\sqrt{w_1(\phi)} A(\phi)^{\top} & 
	\cdots 
	& \sqrt{w_N(\phi)} A(\phi)^{\top} \\[2pt] 
	\sqrt{w_1(\phi)} B(\phi)^{\top} &
	\cdots 
	& \sqrt{w_N(\phi)} B(\phi)^{\top}
	\end{bmatrix}
	=
	\begin{bmatrix}
	\Gamma_{A,i}^{\top} \\ \Gamma_{B,i}^{\top}
	\end{bmatrix} 
	+ \kappa_i 
	\begin{bmatrix}
	\Phi_A^{\top} \\ \Phi_B^{\top}
	\end{bmatrix}.
	\end{align*}
	Hence
	the matrix inequality \eqref{eq:beta_LMI} holds with some $\epsilon > 0$ 
	for $\mu$-almost every $\phi \in \mathcal{M}$ if
	\begin{equation}
	\label{eq:matrix_inequality_with_product_term}
	\begin{bmatrix}
	S_i & 
	\begin{bmatrix}
	I & F_i^{\top}
	\end{bmatrix}
	\left(
	\begin{bmatrix}
	\Gamma_{A,i}^{\top} \\ \Gamma_{B,i}^{\top}
	\end{bmatrix}
	+ \kappa_i 
	\begin{bmatrix}
	\Phi_A^{\top} \\ \Phi_B^{\top}
	\end{bmatrix}
	\right) \bm{S} \\
	\star & \bm{S}
	\end{bmatrix} > 0
	\end{equation} 
	for every $i=1\dots,N$ and for every
	$\begin{bmatrix}
	\Phi_A & \Phi_B
	\end{bmatrix} \in \mathbb{R}^{nN \times (n+m)}$ 
	with
	$\left\|
	\begin{bmatrix}
	\Phi_A & \Phi_B
	\end{bmatrix} 
	\right\| < 1$.
	
	The resulting matrix inequality \eqref{eq:matrix_inequality_with_product_term}
	has the product term of the variables $F_i$ and $\bm{S}$.
	To remove this product term, we employ Lemma \ref{lem:geromel_LMI}.
	Let $U_i \in \mathbb{R}^{n \times n}$ be a nonsingular matrix.
	Defining $R_i := S_i^{-1}$ and $\bar F_i := F_i U_i$,
	we have from Lemma \ref{lem:geromel_LMI} that
	\begin{align}
	&\begin{bmatrix}
	U_i^{\top} & 0 \\ 0 & {\bm S}^{-1}
	\end{bmatrix}
	\begin{bmatrix}
	S_i & 
	\begin{bmatrix}
	I & F_i^{\top}
	\end{bmatrix}
	\left(
	\begin{bmatrix}
	\Gamma_{A,i}^{\top} \\ \Gamma_{B,i}^{\top}
	\end{bmatrix}
	+ \kappa_i 
	\begin{bmatrix}
	\Phi_A^{\top} \\ \Phi_B^{\top}
	\end{bmatrix}
	\right) \bm{S} \\
	\star & \bm{S}
	\end{bmatrix}
	\begin{bmatrix}
	U_i & 0 \\ 0 & {\bm S}^{-1}
	\end{bmatrix} \notag \\ 
	&\qquad \geq
	\begin{bmatrix}
	U_i + U_i^{\top} - R_i & 
	G_i^{\top} + H_i^{\top} \Phi^{\top}
	\\
	\star & \bm{R}
	\end{bmatrix},
	\label{eq:matrix_with_Phi}
	\end{align}
	where 
	the matrices $\bm R$, $\Phi$, $G_i$, and $H_i$ are defined 
	by
	\begin{gather*}
	\bm R := \diag(R_1,\dots,R_{N}),\quad 
	\Phi :=
	\begin{bmatrix}
	\Phi_A  &   \Phi_B 
	\end{bmatrix} \\
	G_i := 
	\begin{bmatrix}
	\Gamma_{A,i}   & \Gamma_{B,i}
	\end{bmatrix}
	\begin{bmatrix}
	U_i  \\  \bar F_i
	\end{bmatrix},\quad
	H_i := \kappa_i 
	\begin{bmatrix}
	U_i \\ \bar F_i
	\end{bmatrix}.
	\end{gather*}
	
	Finally, we obtain the sufficient LMI condition \eqref{eq:LMI_feedback_design} 
	by removing $\Phi$ from the matrix in the right-hand side of the inequality \eqref{eq:matrix_with_Phi}.
	Since $\|\Phi\| < 1$, it follows that
	for every $\rho_i >0$, we have 
	$\rho_i H_i^{\top}(I - \Phi^{\top}\Phi) H_i \geq 0$. Moreover,
	\begin{align*}
	&\begin{bmatrix}
	U_i + U_i^{\top} - R_i & 
	G_i^{\top} + H_i^{\top}\Phi^{\top} \\
	\star & \bm{R}
	\end{bmatrix}
	-
	\begin{bmatrix}
	\rho_i H_i^{\top}(I - \Phi^{\top}\Phi) H_i & 0 \\
	\star & 0 
	\end{bmatrix}  = V^{\top}_i \Omega_i V_i,
	\end{align*}
	where
	\begin{equation*}
	V_i :=
	\begin{bmatrix}
	I & 0 \\
	\Phi H_i & 0 \\
	0 & I \\
	-H_i & 0
	\end{bmatrix},\qquad
	\Omega_i :=\
	\begin{bmatrix}
	U_i + U_i^{\top} - R_i & 0 & G_i^{\top} & \rho_i H_i^{\top} \\
	\star & \rho_i I & I & 0 \\
	\star & \star & \bm{R} & 0 \\
	\star & \star & \star & \rho_i I
	\end{bmatrix}.
	\end{equation*}
	Since $V_i$ is full column rank, it follows that
	if $\Omega_i > 0$, then $V_i^{\top}\Omega_iV_i >0$ and hence the matrix inequality \eqref{eq:matrix_inequality_with_product_term} holds.
	Note that $\rho_i H_i^{\top}$ in $\Omega_i$ has the product of the variables $\rho_i$ and
	$\begin{bmatrix}
	U_i^{\top} & \bar F_i^{\top}
	\end{bmatrix}$.
	However, using the similarity transformation with $\diag(I,1/\rho_i I, I, 1/\rho_i I)$,
	we see that $\Omega_i$ is similar to
	\begin{align}
	\label{eq:suf_SS_last_matrix}
	\begin{bmatrix}
	U_i + U_i^{\top} - R_i & 0 & G_i^{\top} & H_i^{\top} \\
	\star & \lambda_i I & \lambda_i I & 0 \\
	\star & \star & \bm{R} & 0 \\
	\star & \star & \star & \lambda_i I
	\end{bmatrix},
	\end{align}
	in which $\lambda_i := 1/\rho_i$ and the variables appear in a linear form, 
	and the matrix in \eqref{eq:suf_SS_last_matrix} is the one in the left-hand side of 
	the LMIs \eqref{eq:LMI_feedback_design}.
	Thus if the LMIs \eqref{eq:LMI_feedback_design} hold for all $i=1,\dots,N$, then
	the controller \eqref{eq:controller} with $F_i := \bar F_i U_i^{-1}$
	stochastically stabilizes $(A,B)$.
\end{proof}

The controller obtained in Theorem \ref{thm:design_of_gain}
is assumed to know to which box $\mathcal{B}_i$ the parameter $\phi_k$ belongs for 
each $k\in \mathbb{Z}_+$.
The following result can be used to test stabilizability 
and to obtain a stabilizing controller when no information 
about the delays is available.
\begin{corollary}
	Under the same hypothesis as in Theorem~\ref{thm:design_of_gain},
	if
	there exist positive definite matrices $R_i \in \mathbb{R}^{n \times n}$, 
	(not necessarily symmetric) matrices $U \in \mathbb{R}^{n \times n}$, $\bar F \in \mathbb{R}^{m \times n}$, and
	scalars $\lambda_i > 0$ such that
	the following LMIs hold for all $i=1\dots,N$:
	\begin{equation*}
	\begin{bmatrix}
	U + U^{\top} -R_i & 0 \hspace{-1pt} & U^{\top}\Gamma_{A,i}^{\top} + \bar F^{\top} \Gamma_{B,i}^{\top} & 
	\kappa_i 
	\begin{bmatrix}
	U^{\top} & \bar F^{\top}
	\end{bmatrix}\\
	\star  & \lambda_i I  & \lambda_i I \hspace{-1pt} & 0 \\
	\star  & \star  & \bm{R}  & 0 \\
	\star  & \star  & \star  & \lambda_i I
	\end{bmatrix} > 0,
	\end{equation*}
	where $\bm{R} := \diag(R_1,\dots,R_{N})$, then
	the delay-independent controller $F := \bar F U^{-1}$ stochastically stabilizes
	$(A,B)$.
\end{corollary}
\begin{proof}
	This is an immediate consequence of Theorem \ref{thm:design_of_gain}
	with $U_1 = \dots = U_N = U$ and $\bar F_1 = \dots = \bar F_N = \bar F$. 
\end{proof}

We next see how conservative the proposed gridding method is.
We impose the following three assumptions on discrete-time Markov jump systems.
\begin{assumption}
	\label{assump:lebesgue}
	For all $c \in \mathcal{M}$ and  $\delta >0$, 
	the $\sigma$-finite measure $\mu$ on $\mathcal{M}$ satisfies
	\[
	\mu\big(
	\{\phi \in \mathcal{M}:~\|\phi - c\| < \delta \}
	\big) > 0.
	\]
\end{assumption}
\begin{assumption}
	\label{assump:continuity}
	The functions $A \in {\mathbb H}^{n}_{\sup}$ and
	$B \in {\mathbb H}^{n \times m}_{\sup}$ are continuous.
\end{assumption}
\begin{assumption}
	\label{assump:approximation}
	There exist 
	$S \in \mathbb{H}^{n+}_{\sup}$, $F \in \mathbb{H}^{m\times n}_{\sup}$, and $\epsilon >0$
	such that the Lyapunov inequality \eqref{eq:Lyap_type_SS}	
	holds for $\mu$-almost every $\phi \in \mathcal{M}$. Moreover,
	for every $\epsilon_a, \epsilon_{b} >0$, there exist 
	disjoint boxes $\{\mathcal{B}_i\}_{i=1}^{N}$ 
	whose union is $\mathcal{M}$,
	points $c_i \in \mathcal{B}_i$ ($i=1,\dots,N$), and
	piecewise constant functions $S_a \in \mathbb{H}^{n+}_{\sup}$ 
	and $F_a \in \mathbb{H}^{m\times n}_{\sup}$  defined as in \eqref{eq:S_approximation} and \eqref{eq:controller} 
	such that the following three conditions holds:
	\begin{enumerate}
		\item $\|S-S_a\|_{\infty} < \epsilon_a$,\quad $\|F- F_a \|_{\infty} < \epsilon_a$.
		\item For all $i,j=1,\dots,N$, $w_j(\phi)$ defined by \eqref{eq:beta_def} 
		is continuous at $\phi = c_i$.
		\item For $\mu$-almost every $\phi \in \mathcal{B}_i$ and for every $i=1,\dots,N$,
	\end{enumerate}
	\begin{subequations}
		\label{eq:ABbeta_approximation}
		\begin{align}
		\label{eq:AB_approximation}
		&\left\|
		\begin{bmatrix}
		A(\phi) &
		B(\phi)
		\end{bmatrix}
		-
		\begin{bmatrix}
		A(c_i) &
		B(c_i)
		\end{bmatrix}
		\right\| < \epsilon_b \\
		\label{eq:beta_approximation}
		&\left\|
		\begin{bmatrix}
		\sqrt{w_1(\phi)} &
		\cdots &
		\sqrt{w_N(\phi)} 
		\end{bmatrix}
		-
		\begin{bmatrix}
		\sqrt{w_1(c_i)} &
		\cdots &
		\sqrt{w_N(c_i)} 
		\end{bmatrix}
		\right\| < \epsilon_b.
		\end{align}
	\end{subequations}
\end{assumption}
Assumption \ref{assump:lebesgue} holds for the standard Borel measure.
The functions $A$ and $B$ defined by \eqref{eq:A_B_def} 
satisfy Assumption \ref{assump:continuity}.
The first statement of Assumption \ref{assump:approximation}, 
together with Theorem \ref{thm:iff_conds} and  Proposition \ref{prop:real_case},
implies 
that the pair $(A,B)$ is stochastically stabilizable.
Since we approximate a solution of the Lyapunov inequality \eqref{eq:Lyap_type_SS}	by
piecewise-constant functions in Theorem \ref{thm:design_of_gain}, 
we need 1 of Assumption \ref{assump:approximation}.
We use 2 and 3 of Assumption \ref{assump:approximation} together with 
Assumptions \ref{assump:lebesgue} and \ref{assump:continuity} 
to show that a certain inequality holds at $\phi = c_i$ for every $i=1,\dots,N$.
These assumptions on non-zero measure and continuity are required because the Lyapunov inequality \eqref{eq:Lyap_type_SS}
is assumed to be satisfied only at $\mu$-almost everywhere.

The following proposition shows that 
if Assumptions \ref{assump:lebesgue}--\ref{assump:approximation} hold, then
the presented gridding method will guarantee stochastic stabilizability 
given that approximation errors are sufficiently small.
\begin{proposition}
	\label{thm:necessary_stab}
	If Assumptions \ref{assump:lebesgue}--\ref{assump:approximation} hold, then
	there exist disjoint boxes $\{\mathcal{B}_i\}_{i=1}^{N}$ 
	whose union is $\mathcal{M}$ and points
	$c_i \in \mathcal{B}_i$ ($i=1,\dots,N$) such that 
	the LMIs in \eqref{eq:LMI_feedback_design} are feasible.
\end{proposition}
The proof of Proposition \ref{thm:necessary_stab} can be found in the Appendix.


\subsection{Stochastic Detectability}
Next we study the stochastic detectability of the 
pair $A \in {\mathbb H}^{n}_{\sup}$ and
$C \in {\mathbb H}^{r\times n}_{\sup}$.
In our LQ problem, 
we need to check the stochastic detectability of $(C,\bar A)$ or
$(Q-WR^{-1}W^{\top},\bar A)$
in
\eqref{eq:tilde_A_def}
and \eqref{eq:Cholesky_decomposition}.

Define 
an observer gain
$L \in {\mathbb H}^{n \times r}_{\sup}$ as
the piecewise-constant function:
\begin{equation}
\label{eq:observer_gain}
L(\phi) := L_i \in \mathbb{R}^{n \times r} \qquad \forall \phi \in \mathcal{B}_i,
\end{equation}
where the disjoint boxes $\{\mathcal{B}_i\}_{i=1}^{N}$  are chosen as in the previous subsection.
Note that the positions of the variables $K,L$ are different
between $A+BK$ (stabilizability) and $A+LC$ (detectability).
Moreover, unlike the case of countable-state Markov chains 
(see, e.g., \cite{Costa2005}),
the duality of stochastic stabilizability and
stochastic detectability is not proved yet for the case of
continuous-state Markov chains.
Hence we cannot use Theorem~\ref{thm:design_of_gain} directly, but
the gridding method still provides a sufficient condition
for stochastic detectability in terms of LMIs.
\begin{theorem}
	\label{thm:design_of_observer_gain}
	Let $A \in {\mathbb H}^{n}_{\sup}$ and
	$C \in {\mathbb H}^{r \times n}_{\sup}$.
	For each $i=1,\dots,N$,
	define $w_i$ as in Theorem \ref{thm:design_of_gain}
	and
	\begin{equation*}
	\begin{bmatrix}
	\Upsilon_{A,i} \\ \Upsilon_{C,i}
	\end{bmatrix} := 
	\begin{bmatrix}
	A(c_i) \\ 
	C(c_i) 
	\end{bmatrix},\quad 
	\Upsilon_{w,i} :=
	\begin{bmatrix}
	\sqrt{w_1(c_i)} I  &
	\cdots &
	\sqrt{w_N(c_i)} I
	\end{bmatrix}.
	\end{equation*}
	Assume that, 
	for all $i=1,\dots,N$, 
	scalars $\kappa_{A,i},\kappa_{w,i} > 0$ satisfy
	\begin{subequations}
		\label{eq:AC_beta_bound}
		\begin{align}
		&\left\|
		\begin{bmatrix}
		A(\phi) \\ 
		C(\phi) 
		\end{bmatrix}
		-
		\begin{bmatrix}
		\Upsilon_{A,i} \\ \Upsilon_{C,i}
		\end{bmatrix}
		\right\| 
		\leq \kappa_{A,i} \\
		&\left\|
		\begin{bmatrix}
		\sqrt{w_1(\phi)} I  &
		\cdots  &
		\sqrt{w_N(\phi)} I 
		\end{bmatrix} -
		\Upsilon_{w,i} 
		\right\| \leq
		\kappa_{w,i}
		\end{align}
	\end{subequations}
	for $\mu$-almost every $\phi \in  \mathcal{B}_i$.
	If
	there exists positive definite matrices $S_i\in \mathbb{R}^{n \times n}$, 
	(not necessarily symmetric) matrices $U_i \in \mathbb{R}^{n \times n}$, $
	\bar L_i\in \mathbb{R}^{n \times r}$, and
	scalars $\lambda_{i}, \rho_{i} > 0$ such that
	the following LMIs
	hold for all $i=1,\dots,N$:
	\begin{equation}
	\label{eq:LMI_detectability}
	\begin{bmatrix}
	U_i + U_i^{\top}  &  0 & U_i \Upsilon_{A,i} + \bar L_i \Upsilon_{C,i} & 
	\kappa_{A,i}
	\begin{bmatrix}
	U_i  &  \bar L_i
	\end{bmatrix} & 
	0 & \Upsilon_{w,i} 
	{\bm S} & \rho_{i} I \\
	\star  & \lambda_i I & \lambda_i I & 0 & 0 & 0 & 0 \\
	\star & \star & S_i & 0 & 0 & 0 & 0 \\
	\star & \star & \star & \lambda_{i} I & 0 & 0 & 0 \\
	\star & \star & \star & \star & \rho_{i}I & \kappa_{w,i}{\bm S} & 0 \\
	\star   & \star & \star & \star & \star & {\bm S} & 0 \\
	\star   & \star & \star & \star & \star & \star & \rho_{i} I
	\end{bmatrix}
	> 0,
	\end{equation}
	where $\bm{S} := \diag(S_1,\dots,S_{N})$,
	then the pair $(C,A)$ is
	stochastically detectable by the
	observer gain $L$ in
	\eqref{eq:observer_gain} with $L_i := U_i^{-1} \bar L_i$.
\end{theorem}

\begin{proof}
	As in the proof of Theorem \ref{thm:design_of_gain},
	we see from Theorem \ref{thm:iff_conds} and Proposition \ref{prop:real_case},  that
	$(C,A)$ is stochastically detectable if and only if
	there exist 
	$S \in \mathbb{H}^{n+}_{\sup}$, $L \in \mathbb{H}^{n\times r}_{\sup}$, and $\epsilon >0$
	such that the following Lyapunov inequality 
	holds for $\mu$-almost every $\phi \in \mathcal{M}$:
	\begin{equation}
	\label{eq:Lyap_type_SD}
	S(\phi) - 
	(A(\phi)+L(\phi)C(\phi))^{\top} 
	\left(
	\int_{\mathcal{M}} g(\phi,\ell) S(\ell)\mu(d\ell)
	\right)
	(A(\phi)+L(\phi)C(\phi)) \geq \epsilon I.
	\end{equation}
	We prove that 
	if the LMIs in \eqref{eq:LMI_detectability} are feasible for all $i=1,\dots,N$,
	then the Lyapunov inequality \eqref{eq:Lyap_type_SD} holds with
	$S\in {\mathbb H}^{n+}_{\sup}$ defined as 
	a piecewise-constant matrix function $S$ in \eqref{eq:S_approximation}.
	
	By the definitions of $L$ and $S$, we have from
	the Schur complement formula and \eqref{eq:integral_to_summation} that
	the Lyapunov inequity \eqref{eq:Lyap_type_SD} can be transformed into
	\begin{equation}
	\label{eq:Um_sum_inv}
	\begin{bmatrix}
	\big(\sum_{j=1}^{N} w_j(\phi)S_j\big)^{-1} & A(\phi) +L_iC(\phi) \\
	\star & S_i - \epsilon I
	\end{bmatrix} \geq 0
	\end{equation}
	for $\mu$-almost every $\phi \in \mathcal{B}_i$ and for every $i=1,\dots,N$.
	First we remove the nonlinear term $\big(\sum_{j=1}^{N} w_j(\phi)S_j\big)^{-1}$ in \eqref{eq:Um_sum_inv},
	by using Lemma~\ref{lem:geromel_LMI}.
	If $U_i \in \mathbb{R}^{n\times n}$ is nonsingular, then the matrix in the left-hand side of 
	this inequality is similar to
	\begin{align*}
	\begin{bmatrix}
	U_i & 0 \\
	0 & I
	\end{bmatrix}
	\begin{bmatrix} 
	\big(\sum_{j=1}^{N} w_j(\phi)S_j \big)^{-1} & A(\phi) +L_iC(\phi) \\
	\star & S_i - \epsilon I
	\end{bmatrix} \!
	\begin{bmatrix}
	U_i^{\top} & 0 \\
	0 & I
	\end{bmatrix}.
	\end{align*}
	Using Lemma~\ref{lem:geromel_LMI},
	we see that
	\begin{align*}
	&\begin{bmatrix}
	U_i & 0 \\
	0 & I
	\end{bmatrix}
	\begin{bmatrix}
	\big(\sum_{j=1}^{N} w_j(\phi)S_j \big)^{-1} &　A(\phi) +L_iC(\phi) \\
	\star & S_i - \epsilon I
	\end{bmatrix} 
	\begin{bmatrix}
	U_i^{\top} & 0 \\
	0 & I
	\end{bmatrix} \\
	&\quad \geq
	\begin{bmatrix}U_i
	+ U_i^{\top} - \sum_{j=1}^{N} w_j(\phi)S_j
	& U_i(A(\phi) +L_iC(\phi))  \\
	\star & S_i - \epsilon I
	\end{bmatrix} \\
	&\quad =
	\begin{bmatrix}
	U_i
	+ U_i^{\top} 
	& U_i(A(\phi) +L_iC(\phi))  \\
	\star & S_i - \epsilon I
	\end{bmatrix} -
	\begin{bmatrix}
	\sqrt{w_1(\phi)} I   \\
	\vdots & 0 \\
	\sqrt{w_{N}(\phi)} I 
	\end{bmatrix}^{\top}
	{\bm  S}
	\begin{bmatrix}
	\sqrt{w_1(\phi)} I  & 0 \\
	\vdots & 0 \\
	\sqrt{w_{N}(\phi)} I & 0
	\end{bmatrix},
	\end{align*}
	where $\bm{S} := \diag(S_1,\dots,S_{N})$.
	Therefore,
	from the Schur complement formula,
	the matrix inequality \eqref{eq:Um_sum_inv} holds if
	\begin{equation}
	\label{eq:SD_matrix_inequality_with_phi}
	\begin{bmatrix}
	U_i + U_i^{\top} 
	& U_i(A(\phi) +L_iC(\phi))  &
	{\bm S}_{w}(\phi)\\
	\star & S_i & 0 \\
	\star & \star & \bm S
	\end{bmatrix} > 0,
	\end{equation}
	where ${\bm S}_{w}$ is defined by
	\begin{equation}
	\label{eq:S_beta_def}
	{\bm S}_{w}(\phi) :=
	\begin{bmatrix}
	\sqrt{w_1(\phi)} I   &
	\cdots  &
	\sqrt{w_{N}(\phi)} I 
	\end{bmatrix}
	\bm S\qquad 
	\forall \phi \in \mathcal{M}.
	\end{equation}
	
	
	Let us next discretize the matrix inequality \eqref{eq:SD_matrix_inequality_with_phi},
	by using the inequalities \eqref{eq:AC_beta_bound}. In other words,
	we 
	show that if the LMIs \eqref{eq:LMI_detectability} are feasible
	for every $i=1,\dots,N$, then the matrix inequality 
	\eqref{eq:SD_matrix_inequality_with_phi} holds for 
	$\mu$-almost every $\phi \in \mathcal{B}_i$ and for every $i=1,\dots,N$.
	From the inequalities  \eqref{eq:AC_beta_bound}, we see that
	for $\mu$-almost every $\phi \in \mathcal{B}_i$, there exist 
	\[\begin{bmatrix}
	\Phi_A \\ \Phi_C
	\end{bmatrix} \in \mathbb{R}^{(n+r) \times n},~\Phi_{w} \in \mathbb{R}^{n \times nN}
	\text{~with~}
	\left\|
	\begin{bmatrix}
	\Phi_A \\ \Phi_C
	\end{bmatrix}
	\right\| < 1,~\|\Phi_w\| < 1
	\] such that
	\begin{gather*}
	\begin{bmatrix}
	A(\phi) \\ 
	C(\phi) 
	\end{bmatrix}
	=
	\begin{bmatrix}
	\Upsilon_{A,i} \\ \Upsilon_{C,i}
	\end{bmatrix}
	+ \kappa_{A,i} 
	\begin{bmatrix}
	\Phi_A \\ \Phi_C
	\end{bmatrix} \\[4pt]
	\begin{bmatrix}
	\sqrt{w_1(\phi)} I  &
	\cdots  &
	\sqrt{w_{N}(\phi)} I 
	\end{bmatrix}
	=
	\Upsilon_{w,i}  
	+
	\kappa_{w,i} \Phi_{w}.
	\end{gather*}
	Then we obtain
	\begin{align*}
	&
	\begin{bmatrix}
	U_i + U_i^{\top} 
	& U_i(A(\phi) +L_iC(\phi))  & {\bm S}_{w}(\phi) \\
	\star & S_i & 0 \\
	\star & \star & \bm S
	\end{bmatrix} \\
	&\qquad =
	\begin{bmatrix}
	U_i + U_i^{\top} 
	& 
	G_{A,i} + H_{A,i} \Phi
	&
	G_{w,i} + \Phi_{w} H_{w,i} \\
	\star & S_i & 0 \\
	\star & \star & \bm S
	\end{bmatrix},
	\end{align*}
	where, using $\bar L_i := U_iL_i$, we define 
	the matrices $\Phi$, $G_{A,i}$, $H_{A,i}$, 
	$G_{w,i}$, and
	$H_{w,i}$ by
	\begin{gather*}
	\Phi := 
	\begin{bmatrix}
	\Phi_A \\ \Phi_C
	\end{bmatrix},\quad
	G_{A,i} := 
	\begin{bmatrix}
	U_i  &  \bar L_i
	\end{bmatrix}
	\begin{bmatrix}
	\Upsilon_{A,i} \\ \Upsilon_{C,i}
	\end{bmatrix} \\[4pt]
	H_{A,i} := \kappa_{A,i}
	\begin{bmatrix}
	U_i  &  \bar L_i
	\end{bmatrix},\quad 
	G_{w,i}
	:=
	\Upsilon_{w,i} 
	{\bm S},\quad
	H_{w,i}
	:=
	\kappa_{w,i} 
	{\bm S}.
	\end{gather*}	
	Since $\|\Phi\|, \|\Phi_{w}\| < 1$, it follows that
	for all $\rho_i, \rho_{w,i} > 0$, we have
	\[\rho_i H_{A,i} (I - \Phi \Phi^{\top}) H_{A,i}^{\top}  +
	\rho_{w,i}   (I - \Phi_{w} \Phi_{w}^{\top}) \geq 0.\] 
	Moreover,
	\begin{align*}
	&\begin{bmatrix}
	U_i + U_i^{\top} 
	& 
	G_{A,i} + H_{A,i} \Phi
	&
	G_{w,i} + \Phi_{w} H_{w,i} \\
	\star & S_i & 0 \\
	\star & \star & \bm S
	\end{bmatrix} \\
	&\qquad -
	\begin{bmatrix}
	\rho_i H_{A,i} (I - \Phi \Phi^{\top}) H_{A,i}^{\top}  +
	\rho_{w,i}   (I - \Phi_{w} \Phi_{w}^{\top}) 
	& 
	0 
	&
	0 \\
	\star & 0 & 0 \\
	\star & \star & 0
	\end{bmatrix} 
	= V_i^{\top}\Omega_i V_i,
	\end{align*}
	where
	\begin{align*}
	V_i := 
	\begin{bmatrix}
	I & 0 & 0 \\
	\Phi^{\top} H_{A,i}^{\top} & 0 & 0\\
	0 & I & 0 \\
	-H_{A,i}^{\top} & 0 & 0 \\
	\Phi_{w}^{\top} & 0 & 0 \\
	0 & 0 & I \\
	-I & 0 & 0
	\end{bmatrix},\quad
	\Omega_i := 
	\begin{bmatrix}
	U_i + U_i^{\top} \hspace{-4.15pt} &  0 \hspace{-4.15pt}& G_{A,i} \hspace{-4.15pt}& \rho_i H_{A,i} \hspace{-4.15pt}& 
	0 \hspace{-4.15pt}& G_{w,i} \hspace{-4.15pt}& \rho_{w,i} I \\
	\star \hspace{-4.15pt} & \rho_i I \hspace{-4.15pt}& I \hspace{-4.15pt}& 0 \hspace{-4.15pt}& 0 \hspace{-4.15pt}& 0 \hspace{-4.15pt}& 0 \\
	\star \hspace{-4.15pt}& \star \hspace{-4.15pt}& S_i \hspace{-4.15pt}& 0 \hspace{-4.15pt}& 0 \hspace{-4.15pt}& 0 \hspace{-4.15pt}& 0 \\
	\star \hspace{-4.15pt}& \star \hspace{-4.15pt}& \star \hspace{-4.15pt}& \rho_i I \hspace{-4.15pt}& 0 \hspace{-4.15pt}& 0 \hspace{-4.15pt}& 0 \\
	\star \hspace{-4.15pt}& \star \hspace{-4.15pt}& \star \hspace{-4.15pt}& \star \hspace{-4.15pt}& \rho_{w,i}I \hspace{-4.15pt}&  H_{w,i} \hspace{-4.15pt}& 0 \\
	\star \hspace{-4.15pt}  & \star \hspace{-4.15pt}& \star \hspace{-4.15pt}& \star \hspace{-4.15pt}& \star \hspace{-4.15pt}& {\bm S} \hspace{-4.15pt}& 0 \\
	\star \hspace{-4.15pt}  & \star \hspace{-4.15pt}& \star \hspace{-4.15pt}& \star \hspace{-4.15pt}& \star \hspace{-4.15pt}& \star \hspace{-4.15pt}& \rho_{w,i} I
	\end{bmatrix}\!\!.
	\end{align*}
	Since $V_i$ is full column rank, $\Omega_i >0$ leads to $V_i\Omega_iV_i >0$, which implies that 
	the matrix inequality \eqref{eq:SD_matrix_inequality_with_phi} is satisfied
	for $\mu$-almost every $\phi \in \mathcal{B}_i$.
	Note that $\rho_i H_{A,i}$ has the product of the variables $\rho_i$ and
	$\begin{bmatrix}
	U_i & \bar L_i
	\end{bmatrix}$.
	However, applying
	the similarity transformation with
	${\rm diag}(I,1/\rho_i I,I,1/\rho_i I,I, I,I)$,
	we see that $\Omega_i$ is similar to the following matrix:
	\begin{equation*}
	\begin{bmatrix}
	U_i + U_i^{\top}  &  0 & G_{A,i} &  H_{A,i} & 
	0 & G_{w,i} & \rho_{w,i} I \\
	\star  & \lambda_{i} I & \lambda_{i}I & 0 & 0 & 0 & 0 \\
	\star & \star & S_i & 0 & 0 & 0 & 0 \\
	\star & \star & \star & \lambda_{i} I & 0 & 0 & 0 \\
	\star & \star & \star & \star & \rho_{w,i}I &  H_{w,i} & 0 \\
	\star   & \star & \star & \star & \star & {\bm S} & 0 \\
	\star   & \star & \star & \star & \star & \star & \rho_{w,i} I
	\end{bmatrix},
	\end{equation*}
	where  $\lambda_{i} := 1/\rho_i$, and
	this matrix is the one in the left-hand side of \eqref{eq:LMI_detectability}.
	Thus, if
	the LMIs~\eqref{eq:LMI_detectability} are feasible,
	then $(C,A)$ is stochastically detectable.
\end{proof}

As in the case of stochastic stabilizability, 
we see that the proposed gridding method does not
introduce conservatism if approximation errors are sufficiently small.
\begin{assumption}
	\label{assump:continuity_detect}
	The functions $A \in {\mathbb H}^{n}_{\sup}$ and
	$C \in {\mathbb H}^{r \times n}_{\sup}$ are continuous.
\end{assumption}
\begin{assumption}
	\label{assump:approximation_detect}
	There exist 
	$S \in \mathbb{H}^{n+}_{\sup}$, $L \in \mathbb{H}^{n\times r}_{\sup}$, and $\epsilon >0$
	such that the Lyapunov inequality \eqref{eq:Lyap_type_SD}	
	holds for $\mu$-almost every $\phi \in \mathcal{M}$. Moreover,
	for every $\epsilon_a, \epsilon_b >0$, there exist 
	disjoint boxes $\{\mathcal{B}_i\}_{i=1}^{N}$ 
	whose union is $\mathcal{M}$, points $c_i \in \mathcal{B}_i$ ($i=1,\dots,N$), and
	piecewise constant functions $S_a \in \mathbb{H}^{n+}_{\sup}$ 
	and $L_a \in \mathbb{H}^{n\times r}_{\sup}$  defined by \eqref{eq:S_approximation} and \eqref{eq:observer_gain}
	such that the following three conditions holds:
	\begin{enumerate}
		\item $\|S-S_a\|_{\infty} < \epsilon_a$,\quad $\|L- L_a \|_{\infty} < \epsilon_a$.
		\item For all $i,j=1,\dots,N$, $w_j(\phi)$ defined by \eqref{eq:beta_def} 
		is continuous at $\phi = c_i$.
		\item For $\mu$-almost every $\phi \in \mathcal{B}_i$ and for every $i=1,\dots,N$,
		the following inequality and \eqref{eq:beta_approximation} are satisfied:
		\[
		\left\|
		\begin{bmatrix}
		A(\phi) \\
		C(\phi)
		\end{bmatrix}
		-
		\begin{bmatrix}
		A(c_i) \\
		C(c_i)
		\end{bmatrix}
		\right\| < \epsilon_b.
		\]
	\end{enumerate}
\end{assumption}

\begin{proposition}
	\label{thm:necessary_detec}
	If Assumptions \ref{assump:lebesgue}, \ref{assump:continuity_detect}, and \ref{assump:approximation_detect} hold, then
	there exist disjoint boxes $\{\mathcal{B}_i\}_{i=1}^{N}$ 
	whose union is $\mathcal{M}$ and points
	$c_i \in \mathcal{B}_i$ ($i=1,\dots,N$) such that 
	the LMIs in \eqref{eq:LMI_detectability} are feasible.
\end{proposition}
The proof of Proposition \ref{thm:necessary_detec} can be found in the Appendix.

\section{Numerical Example}
Consider the unstable batch reactor studied 
in \cite{Rosenbrock1972}, where
the system matrices $A_c$ and $B_c$ in the continuous-time plant \eqref{eq:plant_dynamics} 
are given by
\begin{align*}
A_c := 
\begin{bmatrix}
1.38 & -0.2077 & 6.715 & -5.676 \\
-0.5814 & -4.29 & 0 & 0.675 \\
1.067 & 4.273 & -6.654 & 5.893 \\
0.048 & 4.273 & 1.343 & -2.104
\end{bmatrix},\qquad
B_c :=
\begin{bmatrix}
0 & 0 \\
5.679 & 0 \\
1.136 &-3.146 \\
1.136 & 0 
\end{bmatrix}.
\end{align*}
This model is widely used as a benchmark example.
We take the sampling period $h = 0.2$ and the delay interval $[\tau_{\min}, \tau_{\max}] = [0,0.03]$.

We consider that the latest delay $\tau_k$ is
stochastically determined by 
the average of the last two delays $\tau_{k-1}$
and $\tau_{k-2}$. More precisely,
the sequence 
$\{ \phi_k:k \in \mathbb{Z}_+\}$, where
$
\phi_k := 
\begin{bmatrix}
\tau_k \\ \tau_{k-1}
\end{bmatrix}
$,
is a Markov chain, and its
transition probability kernel $\mathcal{G}$ is given in the following way:
For every box $\mathcal{B} = [b_{11}, b_{12}] \times [b_{21}, b_{22}]$,
\begin{equation}
\label{eq:delay_example}
\mathcal{G}\left( \begin{bmatrix}
d_1 \\ d_2
\end{bmatrix}\!, \mathcal{B}  \right) = 
\begin{cases}
\frac{\Phi_{d_{\text{ave}}}(b_{12}) - \Phi_{d_{\text{ave}}}(b_{11})}
{\Phi_{d_{\text{ave}}}(\tau_{\max}) - \Phi_{d_{\text{ave}}}(\tau_{\min})} &
\text{if $d_1 \in [b_{21}, b_{22}]$} \\
0 & \text{otherwise},
\end{cases}
\end{equation}
where $d_{\text{ave}} := (d_1+d_2)/2$ and $\Phi_d(x)$ is the probability distribution function of the normal distribution with
mean $d$ and standard deviation $\sigma$.
Fig.~\ref{fig:sample_path} illustrates a sample path of the delay $\tau(t)$
with the initial data $\tau_0 =\tau_{-1} =  0.02$ and the standard derivation $\sigma = 1/100$,
where $\tau(t)$ is defined by 
\[
\tau(t) := \tau_k\qquad \forall t \in [kh+\tau_k, (k+1)h+\tau_{k+1}).
\] 

\begin{figure}[tb]
	\centering
	\includegraphics[width = 7cm]{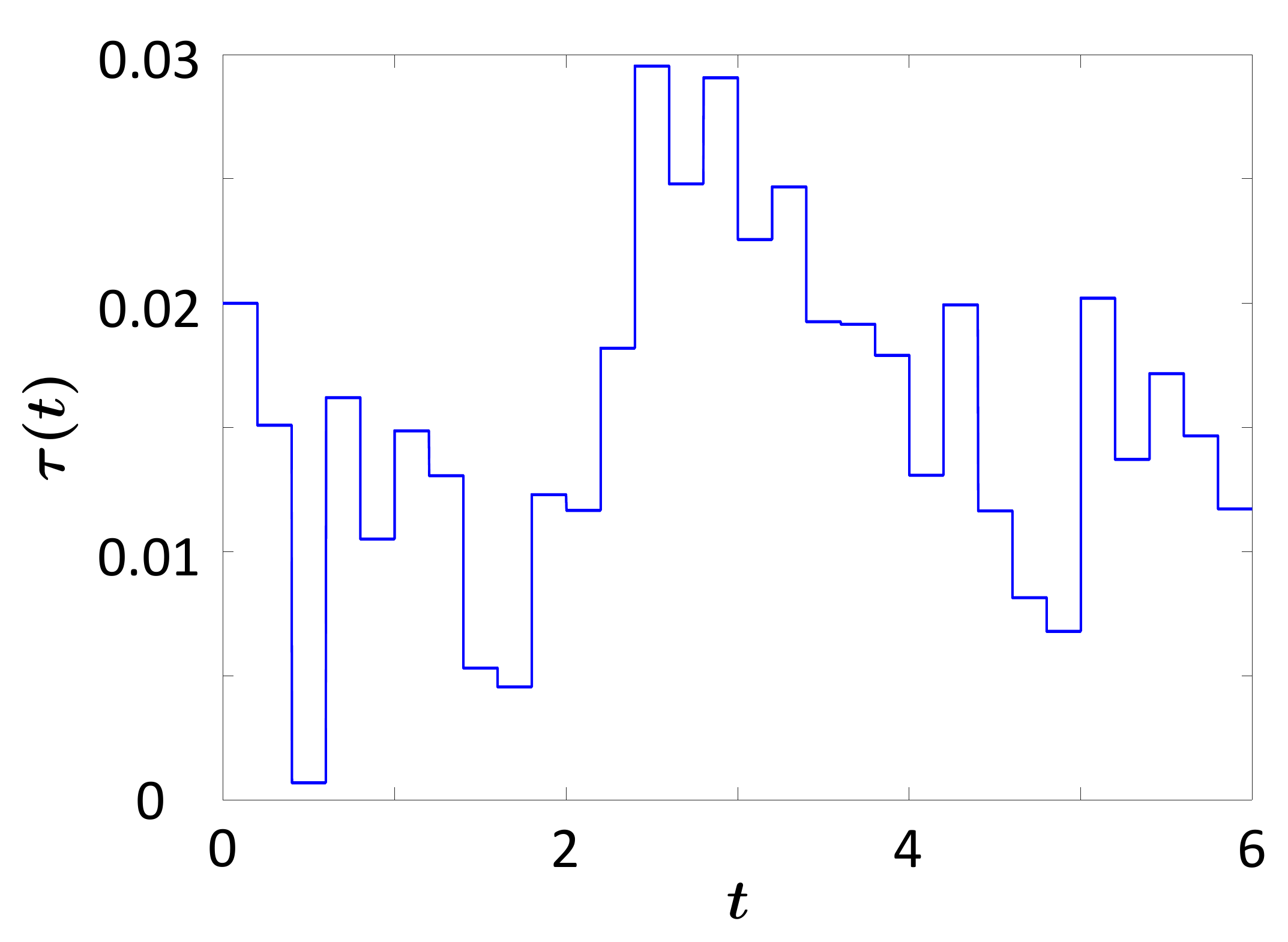}
	\caption{Sample path of $\tau(t)$.}
	\label{fig:sample_path}
\end{figure}

The weighting matrices $Q_c$, $R_c$ for the state and the input in \eqref{eq:Jc_u}
are the identity matrices with compatible dimensions.
Using Theorem \ref{thm:design_of_gain}, we can confirm that
$(A,B)$ in \eqref{eq:A_B_def} 
is stochastically stabilizable. Additionally, 
$(Q-WR^{-1}W^{\top},\bar A)$ in Lemma
\ref{lem:reduction_DT_LQ_cost}
is stochastically detectable by
Theorem \ref{thm:design_of_observer_gain}. Hence, by
Theorems~\ref{thm:reduction} and \ref{thm:LQ_problem_MJS}, we can derive
an optimal controller $u^{\text{opt}}$ 
from the iteration of a Riccati difference equation.

Time responses
are computed for an deterministic initial
state 
\[
x(0) = \begin{bmatrix}1 \\ -1 \\2 \\-2 \end{bmatrix},\quad 
u_{-1} = \begin{bmatrix}0 \\ 0 \end{bmatrix}.
\]
Fig.~\ref{fig:time_response} depicts ten sample paths of the performance function
$\|x(t)\|^2+\|u(t)\|^2$, where
initial delays $\tau_0, \tau_{-1}$ are uniformly distributed in the interval 
$[\tau_{\min},\tau_{\max}]$ and
the standard deviation $\sigma$ of the probability distribution function $\Phi_d$ in \eqref{eq:delay_example}
is given by $\sigma = 1/100$.

We observe that the time responses
with small initial delays are similar to
the response with no delays
by the conventional discrete-time LQ regulator with the same weighting matrices. 
Although larger delays degrade the control performance,
the optimal controller achieves 
almost the same decrease rate for every initial delay.

\begin{figure}[tb]
	\centering
	\includegraphics[width = 7cm]{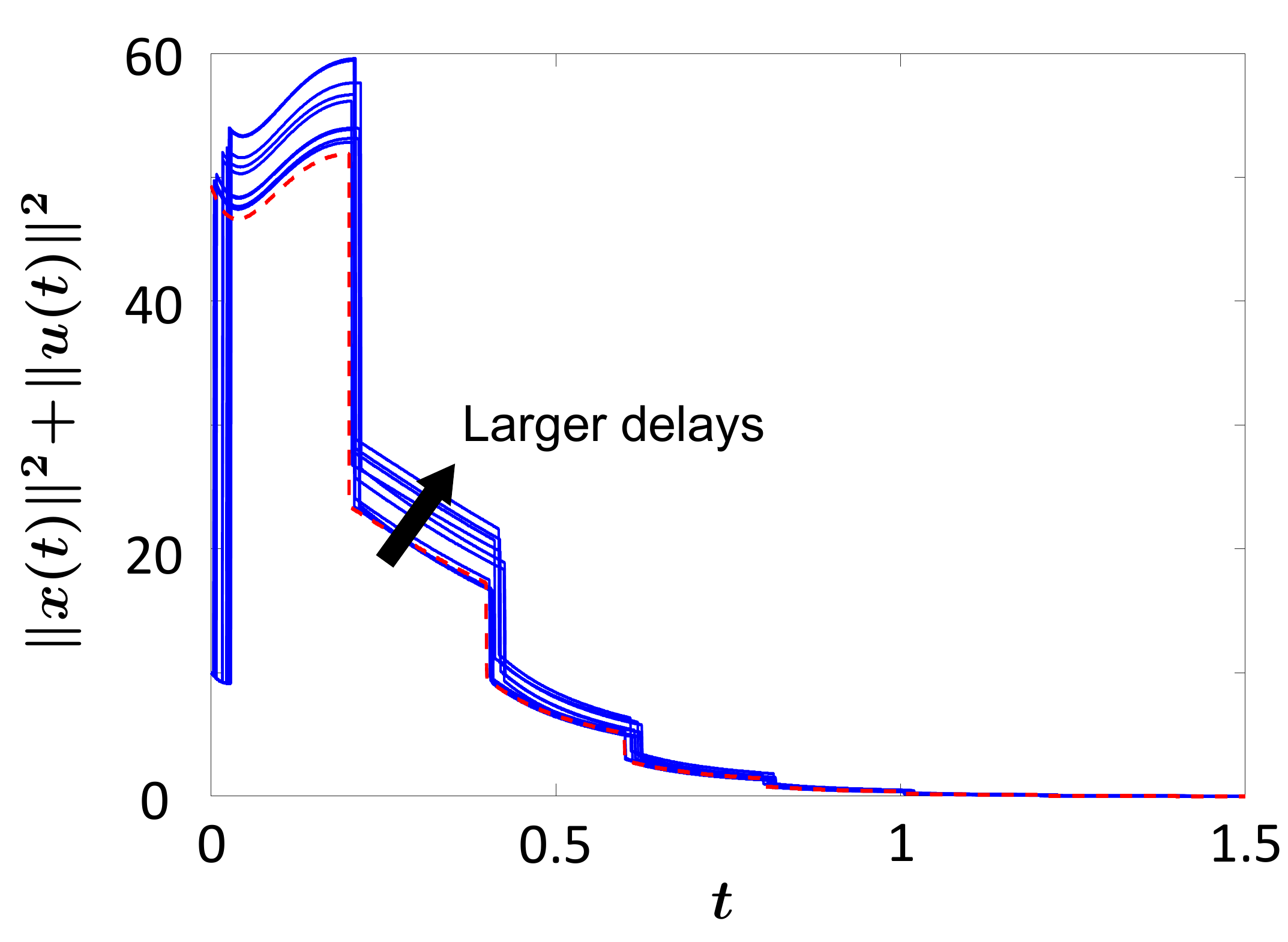}
	\caption{Ten sample paths of $\|x(t)\|^2 + \|u(t)\|^2$:
		The solid lines are the time responses with stochastic delays.
		The dotted line is the time response with no delays, for which we used
		the (conventional) discrete-time LQ regulator computed
		with the same weighting matrices.}
	\label{fig:time_response}
\end{figure}

\section{Concluding Remarks}
We provided the design of delay-dependent optimal controllers
for sampled-data systems whose
sensor-to-controller delays are stochastically determined by the last few delays.
Our optimal control problem was reduced to
the LQ problem of discrete-time Markov jump systems.
We can efficiently compute an optimal controller by
iteratively solving a Riccati difference equation.
Moreover, we derived the sufficient conditions for stochastic stabilizability
and detectability in terms of LMIs via the gridding approach.
From these conditions, we can also construct stabilizing controllers
and state observers.
Future work will focus on addressing more general systems
by incorporating packet losses and
output feedback.

To solve the stabilization problem of discrete-time systems
with continuous-valued Markovian jumping parameters,
we approximated the function $S$ in the
Lyapunov inequality \eqref{eq:Lyap_type_iff} by
a piecewise constant function.
Another interesting study is to approximate the system
\[\{(A(\phi),B(\phi)): \phi \in \mathcal{M}\}\]
by 
\[\{(A(c_i),B(c_i)): \text{$c_i$ is the center of $\mathcal{B}_i$},~~i=1,\dots,N\}\]
and then to consider
the stabilization problem of systems having
discrete-valued Markovian jumping parameters but time-varying uncertainty in the coefficient matrices.

\appendix
\section{Proof of Proposition \ref{thm:necessary_stab}}
First, we prove that 
if $\epsilon_a >0$ in Assumption \ref{assump:approximation} is 
sufficiently small, then
there exists $\epsilon_1 > 0$ such that, for $\mu$-almost every $\phi \in \mathcal{M}$, 
\begin{equation}
\label{eq:Sa_Fa_inequality}
\begin{bmatrix}
S_a(\phi) & (A(\phi) + B(\phi)F_a(\phi))^{\top}
\begin{bmatrix}
\sqrt{w_1(\phi)} I \\ \vdots \\ \sqrt{w_N(\phi)} I
\end{bmatrix}^{\top} \\
\star & \bm{S}^{-1}
\end{bmatrix} > \epsilon_1 I,
\end{equation}
where $\bm{S} := \diag(S_1,\dots,S_N)$.

Since 
\begin{equation}
\label{eq:g_integral}
\int_{\mathcal{M}} g(\phi,\ell) \mu(d\ell) = 1
\end{equation}
for all $\phi \in \mathcal{M}$, it follows from 1 of Assumption \ref{assump:approximation}
that 
\[
\left\|
\int_{\mathcal{M}} g(\phi,\ell) (S(\ell) - S_a(\ell) )\mu(d\ell)
\right\|
\leq 
\int_{\mathcal{M}} g(\phi,\ell) \|S(\ell) - S_a(\ell)\| \mu(d\ell)
\leq \epsilon_a
\]
for $\mu$-almost every $\phi \in \mathcal{M}$.
We obtain
\begin{align*} 
\Bigg\|
&\left(S(\phi) - 
(A(\phi)+B(\phi)F(\phi))^{\top} 
\left(
\int_{\mathcal{M}} g(\phi,\ell) S(\ell)\mu(d\ell)
\right)
(A(\phi)+B(\phi)F(\phi)) \right) \\
&-
\left(S_a(\phi) - 
(A(\phi)+B(\phi)F(\phi))^{\top} 
\left(
\int_{\mathcal{M}} g(\phi,\ell) S_a(\ell)\mu(d\ell)
\right)
(A(\phi)+B(\phi)F(\phi)) \right) 
\Bigg\|\\
&\leq 
\left(
1 + \|A+BF\|_{\infty}^2
\right)  \epsilon_a
\end{align*}
for $\mu$-almost every $\phi \in \mathcal{M}$.
Therefore, if we set $\overline{\epsilon}_a > 0$ to be a value with
\[
\overline{\epsilon}_a< \frac{\epsilon}{1 + \|A+BF\|_{\infty}^2},
\]
then $\epsilon_2 := \epsilon - \big(1 + \|A+BF\|_{\infty}^2 \big) \overline{\epsilon}_a >0$
satisfies
\begin{equation}
\label{eq:Sa_Lyap}
S_a(\phi) - 
(A(\phi)+B(\phi)F(\phi))^{\top} 
\left(
\int_{\mathcal{M}} g(\phi,\ell) S_a(\ell)\mu(d\ell)
\right)
(A(\phi)+B(\phi)F(\phi)) > \epsilon_2 I
\end{equation}
for $\mu$-almost every $\phi \in \mathcal{M}$ and
for every $\epsilon_a \in (0,\overline{\epsilon}_a )$.
Substituting \eqref{eq:integral_to_summation} and \eqref{eq:summation_to_quadratic} 
into \eqref{eq:Sa_Lyap} and applying the Schur complement formula,
we have that
there exists $\epsilon_3 >0$ such that 
\[
\begin{bmatrix}
S_a(\phi) & (A(\phi) + B(\phi)F(\phi))^{\top}
\bm{w}(\phi)\\
\star & \bm{S}^{-1}
\end{bmatrix} > \epsilon_3 I
\]
for $\mu$-almost every $\phi \in \mathcal{M}$,
where 
\[
\bm{w}(\phi) := 
\begin{bmatrix}
\sqrt{w_1(\phi)} I & \cdots & \sqrt{w_N(\phi)} I
\end{bmatrix}.
\]

We obtain
\begin{align*}
&	\begin{bmatrix}
S_a(\phi) & (A(\phi) + B(\phi)F_a(\phi))^{\top}
\bm{w}(\phi) \\
\star & \bm{S}^{-1}
\end{bmatrix} \\
&\quad =
\begin{bmatrix}
S_a(\phi)  & (A(\phi) + B(\phi)F(\phi))^{\top}
\bm{w}(\phi) \\
\star & \bm{S}^{-1}
\end{bmatrix} 
-
\begin{bmatrix}
0& (F(\phi) - F_a(\phi))^{\top}	B(\phi)^{\top}
\bm{w}(\phi) \\
\star & 0
\end{bmatrix}.
\end{align*}
By \eqref{eq:g_integral},
$\{w_i\}_{i=1}^{N}$ satisfies
\begin{equation}
\label{eq:beta_property}
\sum_{i=1}^{N} w_i(\phi) = 
\sum_{i=1}^{N} \int_{\mathcal{B}_i} g(\phi,\ell) \mu(d\ell) = 
\int_{\mathcal{M}} g(\phi,\ell) \mu(d\ell)=
1\qquad \forall \phi \in \mathcal{M}.
\end{equation}
It follows that 
$
\left\|
\bm{w}(\phi)
\right\| = 1
$ for every choice of disjoint boxes $\{\mathcal{B}_i\}_{i=1}^N$ and 
for every $\phi \in \mathcal{M}$. Hence,
if $\|F - F_a\|_{\infty} < \epsilon_a$, then
\[
\left\|
\begin{bmatrix}
0& (F(\phi) - F_a(\phi))^{\top}	B(\phi)^{\top}
\bm{w}(\phi) \\
\star & 0
\end{bmatrix} 
\right\| \leq 
\|B\|_{\infty}\epsilon_a
\]
for $\mu$-almost every $\phi \in \mathcal{M}$.
If $\epsilon_a \in(0,\overline{\epsilon}_a )$ satisfies $\epsilon_a < \epsilon_3/\|B\|_{\infty}$, then
the desired inequality \eqref{eq:Sa_Fa_inequality} holds 
with $\epsilon_1 := \epsilon_3 - \|B\|_{\infty}\epsilon_a$.

Let us next derive the feasibility of the LMIs  in \eqref{eq:LMI_feedback_design} 
from the inequality
\eqref{eq:Sa_Fa_inequality}.
Using the similarity transformation $\diag(S_a^{-1}, I)$,
we find that there exists $\epsilon_4 > 0$ such that 
\begin{equation*}
\begin{bmatrix}
R_i & 
\begin{bmatrix}
R_i & \bar F_i^{\top}
\end{bmatrix}
\begin{bmatrix}
A(\phi)^{\top} \\
B(\phi)^{\top} 
\end{bmatrix}
\bm{w}(\phi) \\
\star & \bm{R}
\end{bmatrix} > \epsilon_4 I
\end{equation*}
for $\mu$-almost every $\phi \in \mathcal{B}_i$ and for every $i=1,\dots,N$, 
where $R_i := S_i^{-1}$, $\bm{R} := \bm{S}^{-1}$ and $\bar F_i := F_iR_i$.
Assumption \ref{assump:lebesgue} on the 
non-zero property of $\mu$ and 
Assumptions \ref{assump:continuity}  and \ref{assump:approximation} on 
the continuity of $A,B,w$ at $\phi=c_i$
imply
that for every $\epsilon_5 \in (0, \epsilon_4)$,
\begin{equation}
\label{eq:Ri_ineq}
\begin{bmatrix}
R_i     
& R_i^{\top}\Gamma_{A,i}^{\top}   +   \bar F_i^{\top} \Gamma_{B,i}^{\top} \\
\star &  \bm{R} 
\end{bmatrix} > \epsilon_5I\qquad \forall i=1,\dots,N,
\end{equation}
where
$\Gamma_{A,i}$ and $\Gamma_{B,i}$ are defined as in Theorem \ref{thm:design_of_gain}.

By 3 of Assumption \ref{assump:approximation}, we see that 
for $\mu$-almost every $\phi \in \mathcal{B}_i$ and
for every $i=1,\dots,N$,
\begin{align}
\label{eq:kappa_cond_nec}
\left\|
\bm{w}(\phi)^{\top}
\begin{bmatrix}
A(\phi) &
B(\phi)
\end{bmatrix}
-
\begin{bmatrix}
\Gamma_{A,i} &  \Gamma_{B,i}
\end{bmatrix}
\right\| 
\leq \big(1 + \left\| \begin{bmatrix} A & B \end{bmatrix} \right\|_{\infty}
\big)\epsilon_b =: \epsilon_{\kappa}.
\end{align}
On the other hand, the the LMI in \eqref{eq:LMI_feedback_design} with $U_i = R_i$ and 
$\kappa_i = \epsilon_{\kappa}$
is equivalent to
\[
\begin{bmatrix}
R_i     
& R_i^{\top}\Gamma_{A,i}^{\top}   +   \bar F_i^{\top} \Gamma_{B,i}^{\top} & 
0 &
\epsilon_{\kappa}
\begin{bmatrix}
R_i^{\top}      & \bar F_i^{\top}
\end{bmatrix}\\
\star   & \bm{R}  & \lambda_i I   & 0 \\
\star   & \star   & \lambda_i I    & 0 \\
\star   & \star   & \star   & \lambda_i I
\end{bmatrix}   >   0.
\]
By the Schur complement formula,
if $\bar F_i = F_iR_i$, then
the above inequality is equivalent to
\begin{equation}
\label{eq:Ri_schur_complement}
\begin{bmatrix}
R_i     
& R_i^{\top}\Gamma_{A,i}^{\top}   +   \bar F_i^{\top} \Gamma_{B,i}^{\top}  \\
\star &  \bm{R} 
\end{bmatrix}
-
\begin{bmatrix}
(\epsilon_{\kappa}^2/\lambda_i) \cdot
R_i^{\top}\big(I + F_i^{\top} F_i\big) R_i
&0 \\
\star &
\lambda_i I
\end{bmatrix}
>0.
\end{equation}
From \eqref{eq:Sa_Lyap}, we know that 
\[
\|R_i\| = \|S_i^{-1}\|  < \frac{1}{\epsilon_2}\qquad \forall i=1,\dots,N.
\]
Therefore, if $\lambda_i > 0$ and $\epsilon_{\kappa} > 0$ satisfy
\[
\lambda_i < \epsilon_5,\qquad 
\epsilon_{\kappa} < \epsilon_2
\sqrt{
	\frac{\lambda_i \epsilon_5}{1+ (\|F\|_{\infty} + \epsilon_a)^2} }
\qquad \forall i = 1,\dots,N,
\]
then the inequality \eqref{eq:Ri_ineq} leads to
the desired conclusion \eqref{eq:Ri_schur_complement}.
If we choose sufficiently small $\epsilon_b >0$, then  $\epsilon_{\kappa} >0$ satisfies the above inequality.
This completes the proof.

\section{Proof of Proposition \ref{thm:necessary_detec}}
By the same discussion as in the proof of Proposition \ref{thm:necessary_stab},
there exists $\epsilon_1 >0$ such that 
\[
\begin{bmatrix}
\big(\sum_{j=1}^{N} w_j(\phi)S_j\big)^{-1} & A(\phi) +L_iC(\phi) \\
\star & S_i 
\end{bmatrix}  > \epsilon_1 I
\]	
for $\mu$-almost every $\phi \in \mathcal{B}_i$ and for every $i=1,\dots,N$.
Using the similarity transformation $\diag\big(\sum_{j=1}^{N} w_j(\phi)S_j, I\big)$ and
applying the Schur complement formula,
we find that there exists $\epsilon_2 > 0$ such that 
\begin{equation*}
\begin{bmatrix}
2U(\phi)
& U(\phi)(A(\phi) + L_iC(\phi))  &
{\bm S}_{w}(\phi)\\
\star & S_i & 0 \\
\star & \star & \bm S
\end{bmatrix}> \epsilon_2 I
\end{equation*}
for $\mu$-almost every $\phi \in \mathcal{B}_i$ and for every $i=1,\dots,N$, 
where $U(\phi) := \sum_{j=1}^{N} w_j(\phi)S_j$, $\bm{S} := \diag(S_1,\dots,S_N)$, and
$\bm{S}_{w}$ is defined as in \eqref{eq:S_beta_def}.
Assumption \ref{assump:lebesgue} 
on the non-zero property of $\mu$ and Assumptions 
\ref{assump:continuity_detect} and \ref{assump:approximation_detect}
on
the continuity of $A,C,w$ at $\phi=c_i$ lead to 
\begin{equation}
\label{eq:UL_ineq}
\begin{bmatrix}
2U_i  &  U_i \Upsilon_{A,i} + \bar L_i \Upsilon_{C,i}  & \Upsilon_{w,i} 
{\bm S}  \\
\star  & S_i &  0 \\
\star & \star &   {\bm S} 
\end{bmatrix} > \epsilon_3 I\qquad 
\forall \epsilon_3 \in (0, \epsilon_2),~\forall i=1,\dots,N,
\end{equation}
where $\Upsilon_{A,i}$, $\Upsilon_{C,i}$, $\Upsilon_{w_i}$ are defined as in 
Theorem \ref{thm:design_of_observer_gain} and
\begin{equation}
\label{eq:U_i_def}
U_i := U(c_i) = \sum_{j=1}^{N} w_j(c_i)S_j,\quad \bar L_i :=U_iL_i.
\end{equation}

The LMIs \eqref{eq:LMI_detectability} with 
$\kappa_{A,i} = \epsilon_b = \kappa_{w,i}$ are equivalent to
\begin{equation*}
\begin{bmatrix}
U_i + U_i^{\top} &  U_i \Upsilon_{A,i} + \bar L_i \Upsilon_{C,i}  & \Upsilon_{w,i} 
{\bm S}  & \epsilon_b
\begin{bmatrix}
U_i  &  \bar L_i
\end{bmatrix} & 
0 & 
0 &  \rho_{i} I \\
\star  & S_i &  0 & 0 & 0 & \lambda_i I & 0 \\
\star & \star &   {\bm S} & 0 & \epsilon_b{\bm S} & 0 & 0 \\
\star & \star & \star & \lambda_{i} I & 0 & 0 & 0 \\
\star & \star & \star & \star & \rho_{i}I & 0  & 0 \\
\star   & \star & \star & \star & \star & \lambda_i I  & 0 \\
\star   & \star & \star & \star & \star & \star & \rho_{i} I
\end{bmatrix}
> 0
\end{equation*}
and hence to
\begin{equation*}
\begin{bmatrix}
U_i \!+\! U_i^{\top}  &  U_i \Upsilon_{A,i} \!+\! \bar L_i \Upsilon_{C,i}  & \Upsilon_{w,i} 
{\bm S}  \\
\star  & S_i &  0 \\
\star & \star &   {\bm S} 
\end{bmatrix}
-
\begin{bmatrix}
\frac{\epsilon_b^2}{\lambda_i} (U_i U_i^{\top} \!+\! \bar L_i \bar L_i^{\top})  \!+\! \rho_i I& 0 & 0 \\
0 & \lambda_i I & 0 \\
0 & 0 & \frac{\epsilon_b^2}{\rho_i}  {\bm S}^2
\end{bmatrix}
> 0
\end{equation*}
by the Schur complement formula.
From \eqref{eq:beta_property}, we know that 
$U_i$ and $\bar L_i$ defined by \eqref{eq:U_i_def} satisfy
\begin{align*}
\|U_i\| &\leq \|U\|_{\infty} \leq \sum_{j=1}^N w_j(\phi) \|S_j\| \leq \|S_a\|_{\infty} \leq \|S\|_{\infty} + \epsilon_a \\
\|\bar L_i\| &\leq 
\|U_i\| \cdot \|L_i\| \leq 
(\|S\|_{\infty} + \epsilon_a) \cdot (\|L\|_{\infty} + \epsilon_a)
\end{align*}
for every $i=1,\dots,N$. Moreover, $\bm{S} = \diag(S_1,\dots,S_N)$ satisfies 
$\|\bm S\|\leq \|S\|_{\infty} + \epsilon_a$.
Hence, 
for the matrices $U_i$ and $\bar L_i$ in \eqref{eq:U_i_def} and $\bm{S} = \diag(S_1,\dots,S_N)$,
there exist $\lambda_i >0$, $\rho_i >0$, $\epsilon_b > 0$ such that 
\[
\begin{bmatrix}
\frac{\epsilon_b^2}{\lambda_i} (U_i U_i^{\top} + \bar L_i \bar L_i^{\top})  + \rho_i I& 0 & 0 \\
0 & \lambda_i I & 0 \\
0 & 0 & \frac{\epsilon_b^2}{\rho_i}  {\bm S}^2
\end{bmatrix} < \frac{\epsilon_3}{2} I.
\]
By \eqref{eq:UL_ineq}, 
the desired LMIs in \eqref{eq:LMI_detectability} are satisfied.


\bibliographystyle{siamplain}

\end{document}